\documentclass[12pt,a4paper]{article} 

\usepackage[latin1]{inputenc}  
\usepackage[T1]{fontenc}
\usepackage{lmodern}           
\usepackage{geometry}
\usepackage{amsmath}
\usepackage{amssymb, amsthm, amsxtra, xspace, epsfig, fancyhdr, parskip, paralist, nicefrac}
\usepackage{booktabs}
\usepackage{dcolumn}
\usepackage{color}
\usepackage{scalerel, stackengine}
\usepackage{subfigure}
\usepackage{todonotes}
 \usepackage{multicol,multirow,tabularx}
 \usepackage{fancyhdr}

\newcommand{\id}{{1\hspace{-1mm}{\rm I}}}
\newcommand{\C}{\mathbb{C}}
\newcommand{\N}{\mathbb{N}}
\newcommand{\Z}{\mathbb{Z}}

\newcommand{\R}{\mathbb{R}}

\newcommand{\F}{\mathcal{F}}

\DeclareMathOperator*{\plim}{\textup{P-lim}}

\DeclareMathOperator{\E}{\mathbb{E}}

\newcommand{\supp}{\textup{supp}}
\newcommand{\card}{\textup{card}}
\newcommand{\diam}{\textup{diam}}
\DeclareMathOperator{\sign}{\text{sign}}
\newtheorem{theorem1}{Theorem}[section]

\newtheorem{theorem4}{Theorem}[section]
\newtheorem{le1}[theorem4]{Lemma}
\newtheorem{le2}[theorem4]{Lemma}
\newtheorem{theorem2}[theorem4]{Theorem}
\newtheorem{theorem3}{Theorem}[section]
\newtheorem{theorem5}[theorem4]{Theorem}

\newtheorem{theorem6}{Theorem}[section]
\newtheorem{theoremub}[theorem4]{Theorem}

\newtheorem{theo1}[theorem1]{Theorem}
\newtheorem{theo2}[theorem1]{Theorem}
\newtheorem{theo3}[theorem1]{Theorem}
\newtheorem{theo4}[theorem1]{Theorem}

\newtheorem{cor1}[theorem6]{Corollary}
\newtheorem{cor4}[theorem4]{Corollary}
\newtheorem{rem4}[theorem4]{Remark}
\newtheorem{lem1}[theorem6]{Lemma}

\newtheorem{lem2}[theorem6]{Lemma}

\newtheorem{lem3}[theorem6]{Lemma}

\newtheorem*{problem}{The Inverse Problem}

\newtheorem{remark2}[theorem3]{Remark}
\newtheorem{remark3}[theorem3]{Remark}
\newtheorem{remark4}[theorem3]{Remark}

\newcommand\reallywidehat[1]{\arraycolsep=0pt\relax%
\begin{array}{c}
\stretchto{
  \scaleto{
    \scalerel*[\widthof{\ensuremath{#1}}]{\kern-.5pt\bigwedge\kern-.5pt}
    {\rule[-\textheight/2]{1ex}{\textheight}} 
  }{\textheight} %
}{0.5ex}\\           
#1\\                 
\rule{-1ex}{0ex}
\end{array}
}

\newtheoremstyle{break}
{9pt}				
{9pt}				
{\itshape}	
{}					
{\bfseries}	
{.}					
{\newline}	
{}					
\theoremstyle{break}

\title{\bf{An Inverse Problem for Infinitely Divisible Moving Average Random Fields}}
\author{W. Karcher, S. Roth, E. Spodarev, C. Walk} 
\usepackage{fancyhdr}
\begin{document}
\maketitle
\vspace{-0.5cm}
\begin{center}
	\textit {Ulm University}
\end{center}	

\renewcommand{\abstractname}{Abstract}
\begin{abstract}
Given a low frequency sample of an infinitely divisible moving average random field $\{\int_{\R^d} f(x-t)\Lambda(dx); \
t \in \R^d \}$ with a known simple function $f$, we study the problem of nonparametric estimation of the L\'{e}vy characteristics of the independently scattered 
random measure $\Lambda$. We provide three methods, a simple plug-in approach, a method based on Fourier transforms and an approach involving decompositions with respect to $L^2$-orthonormal bases, which allow  to estimate the L\'{e}vy density of $\Lambda$. For  these methods,
the bounds for the $L^2$-error are given. Their numerical performance is compared in a simulation study.
\end{abstract}

\noindent
Keywords: Infinitely divisible random measure; stationary random field; L\'{e}vy process, moving average; L\'{e}vy density; Fourier transform; Banach fixed--point theorem.
\section{Introduction} \label{sect:Int}
Let $\Lambda$ be a stationary infinitely divisible independently scattered random measure with L\'{e}vy characteristics
$(a_0,b_0,v_0)$, where $a_0 \geq 0$, $b_0 \in \R$ and $v_0$ is a L\'{e}vy density. Let furthermore $X = \{ X(t); \ t \in \R^d \}$ be
a moving average infinitely divisible random field on $\R^d$ defined by 
\begin{equation}\label{eq:introduction}
X(t) = \int_{\R^d} f(x-t) \Lambda(dx), \quad t \in \R^d,
\end{equation}
with L\'{e}vy characteristics $(a_1,b_1,v_1)$, where $f = \sum_{k=1}^n f_k \id_{\Delta_k}$ is a simple function.
Suppose a sample $(X(t_1),\dots,X(t_N))$ from $X$ is available. The problem studied in this paper is the nonparametric estimation of $(a_0,b_0,v_0)$. For any simple function
$f$ with congruent sets $\Delta_k$, $X(t)$ in (\ref{eq:introduction}) has the same distribution as a linear combination of i.i.d. infinitely divisible random variables.
Therefore, existence and uniqueness of a characteristic triplet $(a_0,b_0,v_0)$ with the property that a certain
linear combination of independent random variables with the corresponding infinitely divisible distribution leading to a random
variable with L\'{e}vy characteristics $(a_1,b_1,v_1)$ becomes a characterization problem for such distributions.
For certain distributions, namely the Poisson and the Gaussian one as well as a mixture of both, all possible
distributions for the summands in the linear combination can be described (see e.g. \cite{Linnik73}).
The disadvantage of those characterization theorems is that they do not give any information about the involved parameters 
(expectation and variance of each summand) and so it is not possible to derive sufficient conditions for the existence of a solution in terms
of the kernel function $f$. Therefore, to solve the inverse problem,  we prefer to use concrete relations between the characteristic triplets of 
$X$ and $\Lambda$ (Section 3) given in terms of $f$. 

The recent preprint \cite{BelPanWoern16} covers the case $d=1$ estimating the L\'evy density $v_0$ of the integrator L\'evy process $\{ L_s\}$ of a moving average process  $X(t)=\int_{\R} f(t-s) \, d L_s,$ $t\in\R$. It is assumed that $\E\, L^2_0<\infty$. The estimate is based on the inversion of the Mellin transform of the second derivative of the cumulant of $X(0)$. A uniform error bound as well as  the consistency of the estimate are given. It is not assumed that $f$ is simple, however, main results  are subject to a number of quite restricting integrability assumptions onto $x^2 v_0(x)$ and $f$ as well as mixing properties of $\{ L_s\}$  that are tricky to check. Additionally, the logarithmic convergence rate shown there (cf. \cite[Corollary 1]{BelPanWoern16}) is too slow. 

In our approach,  we develop the ideas of \cite{Karcher} and use Banach fixed--point theorem combined with a recursive iteration procedure (Theorem \ref{th3}) to give sufficient conditions for the existence of a (unique) solution of our (generally speaking, ill--posed) inverse problem $v_1 \mapsto v_0$. We consider simple functions $f$ since
\begin{enumerate}
\item in applications, $f$ is mainly discretely sampled,
\item any  $f\in L^1(\R^d)$ can be approximated in the  $\|\cdot \|_1$--norm by a sequence of simple $f^{(m)}\in L^1(\R^d)$ (attaining a finite number of values) arbitrarily well, 
\item this allows us to use relatively simple arguments in the proofs and to avoid complex assumptions that are not easy to verify,
\item the $L^2$--convergence rate of our estimates of $v_0$ to its true value is $O(N^{-1})$, cf. Corollaries \ref{cor:ErrEst1} and \ref{cor:ErrEst2}.
\end{enumerate}
The case of arbitrary integrable $f$ is considered in our forthcoming paper \cite{GlRothSpo017}.

This paper is organized as follows: Section \ref{sect:Prelim} gives an introduction to the theory of infinitely divisible 
random measures and stochastic integrals as well as a short overview on $m$-dependent and $\phi$-mixing random
fields together with some moment inequalities (cf. Section \ref{subsect:Moment}).  
In Section \ref{sec:fas}, we describe the inverse problem in detail and give formulas for the relationship between
the characteristics $(a_0, b_0, v_0)$ and $(a_1, b_1, v_1)$. In Section \ref{sect:InvProb}, we obtain sufficient conditions for the existence and uniqueness of the solution of the direct problem, i.e. we propose conditions
under which the mapping $(a_0, b_0, v_0) \mapsto (a_1, b_1, v_1)$ is a bijection. It turns out that this holds true if either one
of the coefficients $f_1,\dots,f_n$ dominates all the others or one of them repeats often enough in some sense.  \\
Estimates for the characteristic L\'{e}vy triplet
of $X$ are given in Section \ref{sec:est_g_1} for pure jump infinitely divisible random fields. Here we use the ideas of \cite{neumann}, \cite{gugushvili} and \cite{comte} originally designed  to estimate
 the L\'{e}vy density of L\'{e}vy processes. The main result of this section
is the proof of the upper bound for the $L^2$-error of the proposed estimator without the assumption of independence of 
observations $X(t_1),\dots,X(t_N)$. The estimation error remains of the same structure as
in the L\'{e}vy process case if the random
field $X$ is assumed to be $m$-dependent or $\phi$-mixing. For the ease of reading, long proofs of the results of this section are moved to Appendix.
Section \ref{sec:es} provides three estimation approaches for the density $v_0$ of $\Lambda$. The first method is a simple plug-in
approach. The second one, the Fourier method, is based on the idea of estimating first the Fourier transform of $v_0$ followed
by another plug-in procedure. The last method uses orthonormal bases in the Hilbert space $L^2[-A,A]$, $A>0$, for a representation of the solution $v_0$ of the inverse problem. After approximating $v_0$ by cutting off its expansion, the coefficients can be estimated by solving a system of
linear equations. For all our methods, we propose upper bounds for the $L^2$-estimation error. 
In the last section, the performance of the methods is compared by numerical simulations.


\section{Preliminaries} \label{sect:Prelim}
Introduce some notation that will be used throughout this paper.

By $\mathcal{B}(\R^d)$ we denote the Borel $\sigma$-field on the d-dimensional Euclidean space $\R^d$. 
The Lebesgue measure on $\R^d$ is denoted by $\nu_d$. We briefly write $\nu_d(dx) = dx$ 
if we integrate w.r.t. $\nu_d$ on $\R^d$. The collection of all bounded Borel sets in $\R^d$ 
will be denoted by $\mathcal{E}_0(\R^d)$. 
For any measurable space $(M, \mathcal{M}, \mu)$ we denote by $L^\alpha(M)$, $1 \leq \alpha < \infty$, the space of all
$\mathcal{M}|\mathcal{B}(\R)$-mesurable functions $f:M \rightarrow \R$ with
$\int_{M} |f|^\alpha (x) \mu(dx) < \infty$. Equipped with the norm $||\cdot||_\alpha =
\left( \int_{M} |f|^\alpha (x) \mu(dx) \right)^{1/\alpha}$, $L^\alpha(M)$ becomes a Banach space
and even in the case $\alpha=2$ a Hilbert space with scalar product $\left\langle f,g \right\rangle_\alpha=
 \int_M f(x)g(x)\mu(dx)$, for any $f,g \in L^2(M)$. With  $L^\infty(M)$ (i.e. if $\alpha = \infty$)
we denote the space of all real valued bounded functions on $M$. In case $(M, \mathcal{M}, \mu) = (\R, \mathcal{B}(\R), \nu_1)$ we
denote by 
$$H^{\delta}(\R) = \{ f \in L^2(\R): \ \int_{\R} |\F f|^2 (x)(1+x^2)^{\delta} dx <\infty \}$$
the Sobolev space
of order $\delta > 0$ equipped with the Sobolev norm $||f||_{H^\delta} = ||\F f(\cdot) (1+\cdot^2)^{\delta/2}||_2$,
where $\F$ is the Fourier transform on $L^2(\R)$. For $f \in L^1(\R)$, $\F f$ is defined by
$\F f (x) = \int_\R e^{itx}f(t)dt$, $x \in \R$.
If $(M, \mathcal{M}, \mu) = (\N, 2^\N, \mu)$
or $(M, \mathcal{M}, \mu) = (\{1,\dots,n\}, 2^{\{1,\dots,n\}}, \mu)$, $n \in \N$,
with $\mu$ being the counting measure, then we write as usual $l^\alpha(M)$ instead of $L^\alpha(M)$ and
all integrals above become sums. Throughout the rest of this paper $(\Omega, \mathcal{A}, P)$ denotes
a probability space. Note that in this case $L^\alpha(\Omega)$ is the space of all random variables
with finite $\alpha$-th moment as well as $||X||_\alpha = \left( \mathbb{E}|X|^\alpha \right)^{1/\alpha}$,
if $1 \leq \alpha < \infty$ and $||X||_\alpha = \sup_{\omega \in \Omega} X(\omega)$ if $\alpha = \infty$, 
for any $X \in L^\alpha(\Omega)$. For an arbitrary set $A$ we introduce furthermore the notation 
$\card(A)$ for its cardinality. Let $\supp f=\{  x\in\R^d: f(x)\neq 0\}$ be the support set of a function $f: \R^d\to \R$. Denote by $\diam (A)=\sup\{  \| x-y \|_{\infty}: x,y\in A   \}$ the diameter of a bounded set $A\subset \R^d$.

\subsection{\textbf{ID} Random Measures and Fields} \label{idrmaf}

Recall some definitions and give a brief overview of infinitely divisible (\textbf{ID})
random measures and fields. 

Let $\Lambda = \{\Lambda(A); \ A \in \mathcal{E}_0(\R^d)\}$ be an \textbf{ID} random measure 
on some probability space $(\Omega, \mathcal{A}, P)$, i.e. a random measure such that
\begin{enumerate}
	\item for each sequence $(E_m)_{m\in \N}$ of disjoint sets in $\mathcal{E}_0(\R^d)$ it holds
	      \begin{enumerate}[(a)]
	      	\item $\Lambda(\cup_{m=1}^\infty E_m) =	\sum_{m=1}^\infty \Lambda(E_m)$ a.s.,
		      	  whenever $\cup_{m=1}^\infty E_m \in \mathcal{E}_0(\R^d)$,
	      	\item $(\Lambda(E_m))_{m\in \N}$ is a sequence of independent random variables.
	      \end{enumerate}  
	\item the random variable $\Lambda(A)$ has an \textbf{ID} distribution for any choice 
		  of $A \in \mathcal{E}_0(\R^d)$.  
\end{enumerate}

Due to the infinite divisibility of the random variable $\Lambda(A)$, its characteristic function, which will be
denoted by $\varphi_{\Lambda(A)}$, has a L\'{e}vy-Khintchin representation which will assumed to be
of the form
\begin{equation}\label{eq:1} 
\varphi_{\Lambda(A)}(t) = \exp \left\lbrace \nu_d(A) K(t) \right\rbrace, \quad A \in \mathcal{E}_0(\R^d), 
\end{equation}
with
\begin{equation}\label{eq:K} 
 K(t) = ita_0 - \frac{1}{2} t^2 b_0 + \int\limits_{\R} \left( e^{itx} - 1 - itx \id_{[-1,1]}(x) \right)v_0(x)dx, 
\end{equation}
where $a_0 \in \R$, $0 \leq b_0 < \infty$ and $v_0$ is a L\'{e}vy
density, i.e. $\int_{\R} \min \{1,x^2\}v_0(x)dx < \infty$. The triplet $(a_0,b_0,v_0)$ will be referred to as 
{\it L\'{e}vy characteristic} of $\Lambda$. It uniquely determines the distribution of the process $\Lambda$. A 
general form for the characteristic function of any \textbf{ID} random measure can be found in \cite[p. 456]{Rajput}. The particular structure of the characteristic function in (\ref{eq:1}) means that the random measure 
$\Lambda$ is stationary with control measure $\lambda: \mathcal{B}(\mathbb{\R}) \rightarrow [0,\infty)$ given by
$$ \lambda(A) = \nu_d(A) \left[  |a_0| + b_0 + \int\limits_{\R} \min \{1,x^2\} v_0(x)dx \right], \quad A \in \mathcal{E}_0(\R^d).$$

Now we can define the stochastic integral w.r.t. the \textbf{ID} random measure $\Lambda$. 
\begin{enumerate}
	\item Let $f = \sum_{j=1}^n x_j \id_{A_j}$ be a real simple function on $\R^d$, where $A_j \in \mathcal{E}_0(\R^d)$
	are pairwise disjoint. Then for every $A \in \mathcal{B}(\R^d)$ we define
	$$ \int\limits_{A}f(x)\Lambda(dx) = \sum\limits_{j=1}^n x_j \Lambda(A \cap A_j). $$
	\item A measurable function $f:(\R^d,\mathcal{B}(\R^d))\rightarrow (\R, \mathcal{B}(\R))$ is said to be
	$\Lambda$-integrable, if there exists a sequence $(f^{(m)})_{m \in \mathbb{N}}$ of simple functions as in
	1. such that
	\begin{enumerate}[(a)]
		\item $f^{(m)} \rightarrow f$, $\lambda$-a.e.
		\item for every $A \in \mathcal{B}(\R^d)$, the sequence $\left( \int_{A} f^{(m)}(x)\Lambda(dx) \right)
		_{m \in \mathbb{N}} $ converges in probability as $m \rightarrow \infty$. In this case we set
		$$ \int\limits_{A} f(x) \Lambda(dx) = \plim\limits_{m\rightarrow \infty} \int\limits_{A} f^{(m)}(x)\Lambda(dx).$$
	\end{enumerate}
\end{enumerate}

A useful characterization of $\Lambda$-integrability is given in \cite[Theorem 2.7]{Rajput}.
Now let $\{f(t - \cdot); \ t\in \R^d\}$ be a family of $\Lambda$-integrable functions induced by the
Borel measurable map $f: \R^d \rightarrow \R$. Then we define
the \textbf{ID} moving average random field $X = \{X(t); \ t \in \R^d\}$ by
\begin{equation}\label{eq:marf}
	X(t) = \int\limits_{\R^d} f(t-x)\Lambda(dx), \quad t \in \R^d.
\end{equation}
A random field is called \textbf{ID} if its finite dimensional distributions are \textbf{ID}.
The random field $X$ defined in (\ref{eq:marf}) is stationary and \textbf{ID} and the characteristic 
function of $\varphi_{X(0)}$ of $X(0)$ is given by
$$ \varphi_{X(0)}(u) = \exp \left\lbrace \int\limits_{\R^d} K(uf(s)) \,ds \right\rbrace,  $$
with $K$ given in (\ref{eq:K}). It is easy to see that
\begin{equation}\label{eq:cumulant}
 \int\limits_{\R^d} K(uf(s)) \, ds = i u a_1 - \frac{1}{2}u^2 b_1 + \int\limits_{\R} (e^{iux}-1-iux\id_{[-1,1]}(x))v_1(x)\,dx 
\end{equation} 
with 
\begin{align}
 a_1 & = \int\limits_{\mathbb{R}^d}U(f(s))ds, \quad  b_1 = b_0 \int\limits_{\R^d} f^2(s)\,ds \label{feq1} \\
v_1(x) & =  \int\limits_{S}  \frac{1}{|f(s)|}v_0\left( \frac{x}{f(s)} \right) ds,\label{feq2} 
\end{align} 
where $a_1 \in \R$, $b_1 \geq 0$, $v_1$ is the L\'{e}vy density of $X(0)$, $S = \supp (f) = \{ s \in \R^d: \ f(s) \neq 0 \}$ denotes the support of $f$
and the function $U$ is defined via
\begin{equation}\label{eq:U}
U(u) = u \left( a_0 + \int\limits_{\mathbb{R}} x \left[ \id_{[-1,1]}(ux)- \id_{[-1,1]}(x) \right] v_0(x)dx \right).
\end{equation} 
The triplet $(a_1,b_1,v_1)$ is again referred to as L\'{e}vy characteristic (of $X(0)$) and determines the distribution of 
$X(0)$ uniquely. Note that due to $\Lambda$-integrability of $f$ all integrals above are finite. This immediately implies that
$f \in L^1(\R^d) \cap L^2(\R^d)$.\\
For details on the theory of infinitely divisible measures and fields with spectral representation as well as proofs for
the above stated facts we refer the interested reader to \cite{Rajput}.


\subsection{$m$-Dependent and $\phi$-Mixing Random Fields} \label{subsect:mDep}

A random field $X = \{X(t),\ t \in T\}$, $T\subseteq \R^d$ defined on $(\Omega, \mathcal{A}, P)$
is called $m$\textup{-dependent} if for some $m\in \N$ and any finite subsets $U$ and $V$ 
of $T$ the random vectors $(X(u))_{u\in U}$ and $(X(v))_{v\in V}$ are independent, whenever
\begin{align*}
||u-v||_\infty = \max\limits_{1\leq i \leq d}{|u_i-v_i|}> m,
\end{align*}
for all $u=(u_1,\dots, u_d)^\top\in U$ and $v = (v_1,\dots, v_d)^\top\in V$. Note that a random field $X$ as in (\ref{eq:marf})
is $m$-dependent, if the support $S$ of $f$ is bounded with $m \geq \diam(S)$. 

Besides, we define the notion of $\phi$-mixing random fields. The {\it mixing coefficient} $\phi$ is defined as follows.
For any $U \subset T$, let $\F_U = \sigma(X(t), t\in U)$ be the $\sigma$--field generated by random variables $X(t), t\in U$. Let furthermore $\mathcal{U}$ and $\mathcal{V}$ be two 
sub-$\sigma$-fields of $\mathcal{A}$.   Define
\begin{align} \label{phimix1}
\phi(\mathcal{U},\mathcal{V}) := \sup\bigl\{|P(V|U)-P(V)|: \, \,V \in \mathcal{V},\, U \in \mathcal{U},\, P(U)\neq 0 \bigr\}
\end{align}
and for $k,l,r \in \N$ 
\begin{align}\label{phimix2}
\phi_{k,l}(r) := \sup\bigl\{\phi(\F_{\Gamma_1},\F_{\Gamma_2}): \,\, \text{card}(\Gamma_1)\leq k,\, \text{card}(\Gamma_2)\leq l,\, d(\Gamma_1,\Gamma_2)\geq r\bigr\},
\end{align}
where $d(\Gamma_1,\Gamma_2) := \min\{||i-j||_\infty: i\in\Gamma_1, j\in \Gamma_2\}$ for $\Gamma_1, \Gamma_2 \subset T$. 
A random field $X = \{X(t), t\in T\}$ on $(\Omega, \mathcal{A}, P)$ is called \textit{$\phi$-mixing} 
or \textit{uniform mixing} if
\begin{align}\label{phimixcon}
\underset{r\rightarrow\infty}{\lim}\,\phi_{k,l}(r) = 0
\end{align}
for any $k,l \in \mathbb{N}$. Equation (\ref{phimixcon}) is called \textit{$\phi$-mixing condition}, see e.g. \cite{doukhan} for more details on mixing.

\subsection{Moment and Exponential Inequalities for Random Fields} \label{subsect:Moment}

In the literature, one can find many moment and exponential inequalities for sums of independent and identically distributed random variables, e.g., the classical\textit{ Rosenthal inequality}  \cite{rosenthal} 
or the \textit{Bernstein inequality} \cite{pollard}.

Similar inequalities hold true for random fields. For $i\in \Z^d$ define the set $V_i^1 = \{j\in\Z^d:\, j <_{\text{lex}} i\}$, where $<_{\text{lex}}$ denotes the
lexicographic order. Let $V_i^k = V_i^1\cap \{j\in\Z^d:\, ||i-j||_\infty\geq k\}$ for $k\geq 2$. For $f(X(t))\in L^1(\Omega) $ set for $k\in\N$
\begin{align*}
\E_k[f(X(t))] := \E\bigl[f(X(t))|\F_{V_t^k}\bigr].
\end{align*}
Figure \ref{FigureVt1+Vtk} shows the sets $V_t^1$ and $V_t^k$ for some $t=(t_1,t_2)\in\Z^2$. The following two results can be found in \cite[pp. 12-14]{dedecker}.

\begin{figure} [h!]
	\centering 
	\subfigure[$V_t^1 = \{j\in\Z^2: \,j <_{\text{lex}} t \}$]
	{\includegraphics[width=0.48\textwidth]{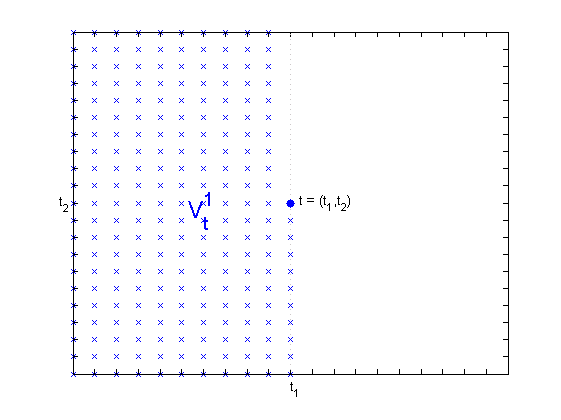} } 	
	\hfill
	\subfigure[$V_t^k = V_t^1\cap \{j\in\Z^2:\, \|t-j\|_{\infty}\geq k\}$]
	{\includegraphics[width=0.48\textwidth]{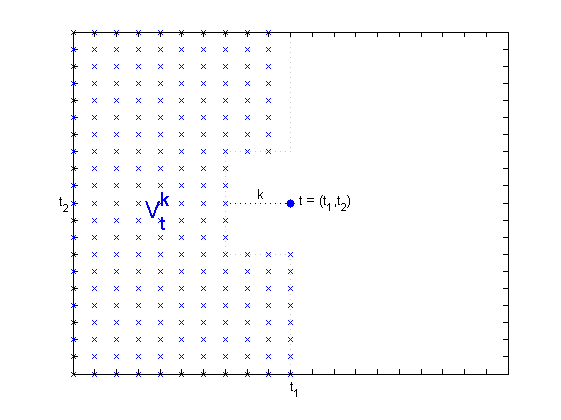} }
	\caption[The sets $V_t^1 = \{j\in\Z^2:\, j <_{\text{lex}} t \}$ and $V_t^k = V_t^1\cap \{j\in\Z^2: \|t-j\|_{\infty}\geq k\}$.]{The sets $V_t^1$ and $V_t^k$ for some $t=(t_1,t_2)\in \Z^2$ and $k\in \N$.} 
	\label{FigureVt1+Vtk}
\end{figure}

\begin{theo1} \label{Th1}
	Let $X = \{X(t), {t\in\Z^d}\}$ be a centered and square-integrable random field. Let $U \subset \Z^d$ be a finite subset. Then for any $p\geq 2$ it holds
	\begin{align*}
	\biggl(\E\biggl|\sum_{t\in U}X(t)\biggr|^p\biggl)^{1/p} \leq \biggl(2p\sum_{t\in U}b_{t,p/2}(X)\biggr)^{1/2},
	\end{align*}
	where $b_{t,\alpha}(X) = \|X(t)^2\|_{\alpha} + \sum_{k\in V_t^1}\|X(k)\E_{ \|k-t\|_{\infty}}[X(t)]\|_{\alpha}$,
	for $t\in U$ and for any $\alpha \geq 1$.
\end{theo1}
\vspace{0.5cm}

\begin{theo2} \label{Th2}
	Let $X = \{X(t), t\in\Z^d\}$ be a field of bounded and centered random variables. Set $b = \sum_{t\in U} b_{t,\infty}(X)$. Then for any positive and real $x$ it holds
	\begin{align*}
	P \Bigl(\Bigl|\sum_{t\in U}X(t)\Bigr|>x\Bigr)\leq \exp \biggl\{\frac{1}{e}-\frac{x^2}{4eb}\biggr\}.
	\end{align*}
\end{theo2}

Note that Theorem \ref{Th1} and Theorem \ref{Th2} are extensions of Burkholder's \cite{hall} and Azuma's \cite{azuma} inequality for martingales.
The next theorem \cite[p. 32]{doukhan} states a Rosenthal-type inequality for $\phi$-mixing random fields.

\begin{theo3}
	\label{Th3}
	Let $X= \{X(t), t\in\Z^d\}$ be a random field. For $p\geq 2$ let $c$ be the smallest even integer such that $c\geq p$. Assume 
	\begin{align}\label{Rosphimix}
	\sum_{r=1}^\infty (r+1)^{d(c-u+1)-1}[\phi_{u,v}(r)]^{1/c}<\infty
	\end{align}
	for all $u,v \in \N$ with $u+v \leq c$, $u,v\geq 2$. Let $U$ be a finite subset 
	of $\Z^d$. If $X(t)$ belongs to $L^p(\Omega)$ and is centered for all $t\in U$, then there exists a positive constant $C$ that depends on $p$ and on 
	the mixing coefficient $\phi_{u,v}(r)$ of $X(t)$ such that 
	\begin{align*}
	\E\biggl|\sum_{t\in U}X(t)\biggr|^p \leq C\cdot \max\biggl\{\sum_{t\in U}\E|X(t)|^p,\Bigl(\sum_{t\in U}\E|X(t)|^2\Bigr)^{p/2}\biggr\}.
	\end{align*}
\end{theo3}

Additionally, the following result can be found in \cite[p. 15]{dedecker}. \\

\begin{theo4} \label{Th4}
	Let $X = \{X(t), t\in \Z^d\}$ be a strictly stationary field of bounded and centered random variables. Take $h\geq \Vert X(0)\Vert_{\infty}$ and set 
\begin{equation}\label{eq:B}	
	B(\phi) = \sum\limits_{j\in\Z^d\backslash 0}\phi_{\infty,1}(|j|)<\infty.
\end{equation}	 
	For any $a_t \in [-1,1]$, $t \in \Z^d$ set $A(U):=\sum_{t\in U}|a_t|$ for $U\subset \Z^d$. For any positive real $x$ we have 
	\begin{align*}
	P\Bigl(\Bigl|\sum_{t\in U}a_tX(t)\Bigr|>x\Bigr)\leq \exp\biggl\{\frac{1}{e}-\frac{x^2}{4(1+B(\phi))A(U)eh^2}\biggr\}.
	\end{align*}
\end{theo4}


\section{Inverse Problem}\label{sec:fas}

In this section, we give a description of the inverse problem treated in this paper. 

Let $\Lambda = \{\Lambda(A), \ A \in \mathcal{E}_0(\R^d) \}$ be a homogeneous \textbf{ID} random measure with
L\'{e}vy characteristics $(a_0,b_0,v_0)$. Consider $f = \sum_{k=1}^n f_k \id_{\Delta_k}$ to be a simple function, where $f_k \in \R \backslash \{0\}$ and
$\Delta_k \in \mathcal{E}_0(\R^d)$ pairwise disjoint, $k=1,\dots ,n$. 
Assume furthermore $X = \{X(t), \ t \in \R^d\}$ to be an \textbf{ID} moving average random field of the form
\begin{equation}\label{eq:X} 
X(t) = \int\limits_{\R^d} f(t-x) \Lambda(dx)= \sum\limits_{k=1}^n f_k \Lambda(t-\Delta_k) , \quad t\in\R^d,
\end{equation}
where $t-A = \{t-x: \ x \in A\} \subset \R^d, \ t\in \R^d$ for an arbitrary set $A$.\\

\begin{problem}
	Given $N \in \mathbb{N}$ observations $X(t_1),\dots,X(t_N)$ at points $t_1,\dots,t_N \in \R^d$ 
	of the random field $X$, estimate the L\'{e}vy triplet $(a_0,b_0,v_0)$ of the \textbf{ID} 
	random measure $\Lambda$.
\end{problem}
Formulas (\ref{feq1}) and (\ref{feq2}) then become 
\begin{align}\nonumber
a_1 &= \sum\limits_{k=1}^n U(f_k)\nu_d(\Delta_k), \quad b_1 = b_0 \sum\limits_{k=1}^n f_k^2 \nu_d(\Delta_k),  \\
v_1(x) &= \sum\limits_{k=1}^n \frac{\nu_d(\Delta_k)}{|f_k|}v_0 \left( \frac{x}{f_k} \right) , \quad x \in \R \backslash \{0\}, \label{eq:v_1}
\end{align}
with $U$ defined in (\ref{eq:U}).
For known $a_1,$ $b_1,$ $v_0$, the  above equations are easily solvable w.r.t. $a_0$ and $b_0$, thus providing an estimation approach for $a_0$ and $b_0$. So, given $v_1$, the main point is now to find
a solution $v_0$ of the last equation. In the next section, we give some sufficient conditions
under which a solution exists and is unique.

\section{Existence and Uniqueness of a Solution for $v_0$} \label{sect:InvProb}

In the following, we assume w.l.o.g. that $\nu_d(\Delta_k)=1$ for all $k=1,\dots,n$. 
Typically it is common to estimate $x^n v_1(x)$ rather than $v_1(x)$ itself, since many of the estimators
for L\'{e}vy densities are based on derivatives of the Fourier transform (in the context of L\'{e}vy processes, 
see e.g. \cite{neumann,gugushvili,comte}). For this purpose let $h:\mathbb{R} \rightarrow \mathbb{R}$
be a measurable function such that  
\begin{equation}\label{eq:h_cond}
\min \{ 1,\cdot^2 \} g(\cdot) / h(\cdot)\in L^1(\R) \mbox{ for  any } g \in L^2(\R),
\end{equation}
\begin{equation}\label{supcondition}
s(y)=\sup\limits_x \{|h(x)| / |h(y x)|)\} < \infty  \mbox{ for  any }  y\neq 0.
\end{equation}
 A sufficient condition for \eqref{eq:h_cond} to hold is
\begin{equation}\label{eq:h_cond1}
\int_{\R} \frac{\min \{ 1,x^4 \}}{h^2(x)} dx < \infty.
\end{equation}
 Indeed, the Cauchy-Schwarz inequality yields
$$ \int_{\R} \min \{ 1,x^2 \} \left| \frac{g(x)}{h(x)} \right| dx \leq \left( \int_{\R} \frac{\min \{ 1,x^4 \} }{h^2(x)} dx \right)^{1/2} ||g||_2 < \infty.$$
Examples of functions $h$ satisfying \eqref{eq:h_cond}--\eqref{eq:h_cond1} are $h(x)=1$, $h(x)=|x|^\beta$, $\beta\in (1/2,5/2)$ and $h(x)=x^\beta$, $\beta=1,2$.
Consider the modified equation
\begin{equation}\label{eq:mainH}
(h v_1)(x) = \sum\limits_{k=1}^n \frac{1}{|f_k|} \frac{h(x)}{h(x/f_k)}(h v_0)(x/f_k).
\end{equation}
It is understood in $L^2(\R)$-sense, where it is assumed that $g^{(h)}_0 = hv_0$ and $g^{(h)}_1 = hv_1$ are both
in $L^2(\R)$.
Let $Q = \{1 \leq k \leq n: \ f_k = f_1\}$ be the set of all indices of coefficients $f_k$ that coincide with
$f_1$. Denote by $n_1 = \text{card}(Q)$ its cardinality. 
Define 
\begin{equation}\label{eq:S_k}
s_k=s(f_1/f_k), \quad k=1,\dots,n.
\end{equation}
The following theorem states conditions, under which equation (\ref{eq:mainH}) has
a unique solution for fixed $g^{(h)}_1 \in L^2(\R)$. 

\begin{theorem3}\label{th3}
Let a function $h: \R \rightarrow \R$ be given as above.
 Then equation (\ref{eq:mainH}) has a unique solution $g^{(h)}_0 \in L^2(\R)$ for
any $g^{(h)}_1 \in L^2(\R)$ if
\begin{equation}\label{eq:mcond}
e(f,h)=\frac{1}{n_1}\sum\limits_{k:f_k \neq f_1}  s_k \cdot \left(  |f_1|/|f_k| \right)^{1/2} < 1. 
\end{equation}
The solution is given by the formula
\begin{equation}\label{eq:solution}
\begin{split}
g^{(h)}_0(\cdot) & = \frac{|f_1|}{n_1}\frac{h(\cdot)}{h(f_1\cdot)}g^{(h)}_1(f_1 \cdot) \\ 
&+ \sum\limits_{j=1}^\infty (-1)^j \!  \!   \! \ \sum\limits_{i_1: f_{i_1}\neq f_1} \!  \!  \!  \! 
\dots  \!  \!  \!  \sum\limits_{i_j: f_{i_j}\neq f_1} \frac{\left( |f_1|/n_1 \right) ^{j+1}}{|f_{i_1}\dots f_{i_j}|} \frac{h(\cdot)}
{h \left( \frac{f_1^{j+1}}{f_{i_1}\dots f_{i_j}}\cdot\right) } 
g^{(h)}_1\left( \frac{f_1^{j+1}}{f_{i_1}\dots f_{i_j}}\cdot \right).
\end{split}	
\end{equation}
\end{theorem3}

\begin{proof}
Let $g^{(h)}_1 \in L^2(\R)$. Define the operator $\varphi_{g^{(h)}_1}:L^2(\R) \rightarrow L^2(\R)$ by
$$ \varphi_{g^{(h)}_1}(r) = \frac{|f_1|}{n_1}\frac{h(\cdot)}{h(f_1 \cdot)}g^{(h)}_1(f_1 \cdot) 
- \sum\limits_{k: f_k \neq f_1} \frac{|f_1|}{n_1|f_k|}\frac{h(\cdot)}{h\left(\frac{f_1}{f_k} \cdot\right)}r\left(\frac{f_1}{f_k}\cdot\right) $$
Then formula (\ref{eq:mainH}) yields a fixed point of $\varphi_{g^{(h)}_1}$, i.e., is a solution of equation 
\begin{equation}\label{eq:fixedpoint}
g^{(h)}_0 = \varphi_{g^{(h)}_1}(g^{(h)}_0).
\end{equation}
 It is straight forward to see that for any functions $u_1,u_2\in L^2(\R)$ it holds
$$ ||\varphi_{g^{(h)}_1}(u_1)-\varphi_{g^{(h)}_1}(u_2)||_{2} \leq e(f,h)
||u_1-u_2||_2, $$
i.e. $\varphi_{g^{(h)}_1}$ is a contraction. By Banach fixed-point theorem there exists a unique solution
$g^{(h)}_0 \in L^2(\R)$ to the equation \eqref{eq:fixedpoint} which shows the first part of the theorem.
Relation \eqref{eq:solution} can easily be obtained by iterating equation \eqref{eq:fixedpoint}  w.r.t. $g^{(h)}_0$. 
\end{proof}
\begin{remark2}
	Note that the choice of $f_1$ in this setting is arbitrary. The statement of Theorem \ref{th3} does not depend on a 
	certain order of the coefficients $f_1,\dots,f_n$. In particular, this means that $f_1$ in the definitions of $Q$ and
	$n_1$ can be replaced by any other coefficient $f_{j_0}$, $j_0 \in \{2,\dots,n\}$. Consequently, substituting 
	$f_1$ by $f_{j_0}$ in Theorem \ref{th3} leads to the same solution $g_0^{(h)}$.
	Indeed, let $f_j$, $j \neq 1$ be any other coefficient that fulfills the
	conditions of Theorem \ref{th3}, and let $\bar{g}_0^{(h)}$ be the corresponding solution of (\ref{eq:mainH}). Then 
	$$ 0 = \sum\limits_{k=1}^n \frac{1}{|f_k|} \frac{h(x)}{h(x/f_k)}
	(g_0^{(h)} - \bar{g}_0^{(h)})(x/f_k) $$
	Due to Theorem \ref{th3}, this equation has a unique solution. Since $0$
	is a solution it thus follows that $g_0^{(h)} - \bar{g}_0^{(h)} = 0$
	(in $L^2(\R)$-sense).
\end{remark2}

\vspace{0.3cm}
\begin{remark3}
	Theorem \ref{th3} gives sufficient conditions for the existence and uniqueness of a solution (\ref{eq:solution}) of equation (\ref{eq:mainH}). If condition (\ref{eq:mcond}) fails to hold, no solution as well as infinitely many solutions of (\ref{eq:mainH}) are possible. One can easily construct corresponding examples illustrating that. Consider e.g. $n=2$, $f_1 = 1$ and $f_2 = -1$. Now choose $h$ to be any odd function satisfying \eqref{eq:h_cond}-\eqref{eq:h_cond1}. Clearly condition (\ref{eq:mcond}) is not fulfilled.   Then (\ref{eq:mainH}) becomes 
	\begin{equation} \label{eq:mainHEx}
	 g_1^{(h)}(x) = g_0^{(h)}(x) + \frac{h(x)}{h(-x)}g_0^{(h)}(-x) = g_0^{(h)}(x) - g_0^{(h)}(-x).
\end{equation}	 
	\begin{enumerate}
		\item Let $g_1^{(h)} \in L^2(\R)$ be any even function, $g^{(h)}_1 \neq 0$ a.e. Then (\ref{eq:mainHEx}) has no solution since its right--hand side is odd.
		\item If, on the other hand, $g_1^{(h)}(x) = 0$ a.e. then any even $L^2$-function
		$g_0^{(h)}$ is a solution of (\ref{eq:mainHEx}).
	\end{enumerate}
	Note that condition (\ref{supcondition}) ensures that $h(\cdot)g_1^{(h)}(f_1\cdot) / h(f_1\cdot) \in L^2(\R)$ for any 
	$g_1^{(h)} \in L^2(\R)$. This condition is necessary. Consider e.g.
	 $g_1^{(h)} (x)= e^{-|x|/2}$, $h(x) = e^{|x|}$, $x \in \R$,
	as well as $f_1 = f_2 = f_3 = 1/4$, $f_4 = 1/16$. Then, except for (\ref{supcondition}), all conditions of Theorem \ref{th3} are fulfilled, but $h(\cdot)g_1^{(h)}(f_1\cdot) / h(f_1\cdot)\not\in L^2(\R)$ in this case. Thus (\ref{eq:solution}) cannot be an $L^2$-solution.
\end{remark3}

\vspace{0.3cm}
\begin{remark4}
Condition (\ref{eq:mcond}) is not necessary for the existence and uniqueness of a solution of equation (\ref{eq:mainH}).
		As a counterexample, consider  $n=3$, $f_1 = e^\alpha$, $f_2 = e^{2\alpha}$, $f_3 = e^{3\alpha}$, and $h(x) = x$.
		If 
		$$ 2\log \left( \frac{-1 + \sqrt{5}}{2} \right)  \leq \alpha \leq 2\log \left( \frac{1 + \sqrt{5}}{2} \right)$$
		then none of the coefficients fulfills (\ref{eq:mcond}). In our paper \cite{GlRothSpo017} we prove necessary and sufficient conditions for existence and uniqueness of a solution of integral equation (\ref{feq2}). It can be shown that $f = \sum_{k=1}^{3} e^{k\alpha} \id_{\Delta_k}$ satisfies those conditions and hence there is a unique solution of (\ref{eq:mainH})  for any $g_1^{(h)} \in L^2(\R)$.
\end{remark4}

Condition (\ref{eq:mcond}) means that one of the coefficients (here $f_1$) dominates all others either in its magnitude $|f_1|$ or in its frequency $n_1$. To illustrate this, consider any power function $h(x) = |x|^\beta$ with $\beta \in (1/2,5/2)$ and $|x|^\beta
v_1(x) \in L^2(\R)$. Then $s_k = 
(|f_k|/|f_1|)^\beta$, $k=1,\dots,n$ and the equation is solvable w.r.t. $|x|^\beta v_0(x)$ if
$$\frac{1}{n_1}\sum\limits_{k: f_k \neq f_1} \left( \frac{|f_k|}{|f_1|} \right)^{\beta - 1/2} < 1.  $$
In particular, if $n_1=1$ this means that $|f_1| > \max \{|f_2|,\dots,|f_n|\}$.
If $h$ is strictly positive and super-homogeneous of degree $\alpha$,  i.e. 
$$ h(cx) \geq c^\alpha h(x), \quad x\in \mathbb{R}$$
for all $c \geq 0$ and some $\alpha > 0$,  then condition (\ref{supcondition}) is fulfilled
if all the coefficients $f_k$ have the same sign. Then (\ref{eq:mcond}) holds if 
$$ \frac{1}{n_1}\sum\limits_{k: f_k \neq f_1} \left( \frac{f_k}{f_1} \right)^{\alpha - 1/2} < 1. $$

	\section{Estimation of $g^{(h)}_1$ for Pure Jump {\bf ID} Random Fields}\label{sec:est_g_1}

Modern statistical literature contains quite a number of methods to estimate the L\'{e}vy density $v_1$ of $X(0)$ if $d=1$, i.e., $X$ is a L\'{e}vy process, see \cite{Antoniadis06,comte,neumann,comte1,gugushvili,Gugushvili15,Gugushvili152},  \cite{Zuyev} and references therein. They range from moment fitting and maximum likelihood ratio to inverse Fourier methods based on the empirical characterstic function of $X(0)$. For simplicity, one often assumes that the drift and the Gaussian part of $X(0)$ vanish, thus letting $X$ be a pure jump   L\'evy process.

In the recent preprint \cite{Zuyev}, the problem of estimation of the L\'evy measure of $X(0)$  was solved for compound Poisson L\'{e}vy processes  $X$ using
variational analysis on the cone of measures and the steepest descent method of minimizing of a certain risk functional implemented for the discrete (atomic) measures. The resulting estimate of $v_1$ can be obtained out of these measures by smoothing.

For all our estimation approaches in the next section, either estimators for $g^{(h)}_1$
or at least for its Fourier transform $\F[g^{(h)}_1]$ are required to proceed with the estimation of $v_0$. Therefore we adopted an estimation procedure from \cite{comte1,neumann}
for pure jump L\'evy processes   to 
estimate $v_1$. The main difference to L\'{e}vy processes is in our case the assumption of independent
increments which obviously is not given for random fields in arbitrary dimension $d$.
Nevertheless,  assuming $X$ to be $m$-dependent or 
$\phi$-mixing allows us to use the same ideas for the estimation of
$g^{(h)}_1$. 

Consider a stationary random field $X$ as in (\ref{eq:X}) with characteristic function $\varphi_{X(0)}(u)$ given by
$$ \psi(u) := \varphi_{X(0)}(u) = \mathbb{E} e^{iuX(0)} = \exp \left\lbrace
\int\limits_{\R} \left( e^{iux}  - 1 \right)v_1(x)dx  \right\rbrace  . $$
Note that its logarithm coincides with formula (\ref{eq:cumulant}) by taking $a_1 = \int_{-1}^1 x v_1(x) dx$ and $b_1 = 0$.
Under the additional assumption $\int_{\R} |x| v_1(x) dx < \infty$ it holds
$$ \psi^\prime (u) = i \psi(u) \int\limits_{\R} e^{iux} x v_1(x)dx = i \psi(u) \F[xv_1](u), $$
that is equivalent to 
\begin{equation}\label{FourierRelation}
 \F[g_1](u) = -i \frac{\psi^\prime(u)}{\psi(u)}, 
\end{equation}
where $g_1(x) := g^{(h)}_1(x) = xv_1(x)$ (taking $h(x) = x$) and $\F[g_1]$ denotes the Fourier transform of $g_1$. 
Now let $X$ be discretely observed on 
a regular grid $\Delta \mathbb{Z}^d$ with mesh size $\Delta > 0$, i.e. we consider the random field
$Y = \{Y_j ;\, j \in \Z^d\},$ where
\begin{equation}\label{eq:rfY}
Y_j = X(\Delta j), \quad \Delta j = (\Delta j_1, \dots, \Delta j_d), \quad
j = (j_1,\dots,j_d) \in \mathbb{Z}^d.
\end{equation}
For a finite nonempty set $W \subset \mathbb{Z}^d$
with cardinality $N = |W|$ let $(Y_{j})_{j \in W}$ be a sample from $Y$. By taking the empirical
counterparts
\begin{align}
\hat{\psi}(u) & = \frac{1}{N} \sum_{j \in W} e^{i u Y_j}, \nonumber \\
\hat{\theta}(u) & = \frac{1}{N} \sum_{j \in W}Y_j e^{i u Y_j}, 
\end{align}
of $\psi(u)$ and $\theta(u):=\psi^\prime(u)$ on the right--hand side of (\ref{FourierRelation}) an estimator for $\F[g_1]$ can
be defined as
\begin{equation}\label{eq:Fg_1_est}
\widehat{\F[g_1]}(u) =  -i \frac{\hat{\theta}(u)}{\tilde{\psi}(u)},
\end{equation}
where 
\begin{equation}\label{stabilizer}
\frac{1}{\tilde{\psi}(u)} := \frac{1}{\hat{\psi}(u)}\id \{{|\hat{\psi}(u)| > N^{-1/2}}\}.
\end{equation}
The indicator function on the right hand side of (\ref{stabilizer}) ensures the stability of the estimator for small values
of $|\hat{\psi}(u)|$. Based on this idea Comte and Genon-Catalot \cite{comte1} provided the estimator 
$$\hat{g}_{1,l}(x) = \frac{1}{2 \pi} \int\limits_{-\pi l}^
{\pi l} e^{-ixu} \widehat{\F[g_1]}(u)du$$
for $g_1$. We make the following assumptions: for a $k\in \N$
\begin{description}
	\item[(H1)] $g_1\in L^1(\R)\cap L^2(\R)$
	\item[${\rm \bf (H2)}_k$] $\int_{\R} |x|^{k-1}|g_1(x)|dx < \infty$
	\item[(H3)] $\exists \ c_\psi, \ C_\psi > 0$ and $\beta \geq 0$ such that for all $x \in \R$
				$$ c_\psi (1+x^2)^{-\beta/2} \leq |\psi(x)| \leq C_\psi (1+x^2)^{-\beta/2} $$
	\item[(H4)] $g_1 \in H^\beta(\R)$ where $\beta>0$ is as in  {\bf (H3)}.		
\end{description}
Assumptions {\bf (H1)}--${\rm \bf (H2)}_k$ are moment conditions for $X(0)$.   Assumptions  {\bf (H3)}-- {\bf (H4)} are used to compute $L^2$--error bounds and 
rates of convergence of L\'evy density estimates, cf. \cite{neumann}.
For the random field $Y$ we define
\begin{align*}
\xi_t^{(1)}(u) & = Y_t\cos(uY_t)-\E \big(Y_0\cos(uY_0)\big), \\ 
\xi_t^{(2)}(u) & = Y_t\sin(uY_t)-\E \big( Y_0\sin(uY_0)\big), \\
\tilde{\xi}_t^{(1)}(u) & = \cos(uY_t)-\E\cos(uY_0), \\
\tilde{\xi}_t^{(2)}(u) & = \sin(uY_t)-\E\sin(uY_0),
\end{align*}
where $u \in \mathbb{R}$, $t\in \Z^d$. Under condition ${\rm \bf (H2)}_2$, it holds $\E X^2(0)<\infty$
and hence $\E \big(\xi_t^{(i)}(u)\big)^2<\infty$,  $\E \big(\tilde{\xi}_t^{(i)}(u)\big)^2<\infty$ for $i=1,2$, $t\in \Z^d$ and $u\in \R$.
Introduce the notation 
$\lVert \xi \rVert_{\cdot}=\left( \E  \lVert \xi \rVert_2^2 \right)^{1/2}$ for any random function $\xi:\Omega\times \R \to \C $ s.t. $\xi\in L^2(\Omega\times \R) $.

The following $L^2$-error bounds for $\hat{g}_{1,l}$ will be proven in Appendix.
\begin{theorem6}\label{TheoGenCat}
	Assume that \textbf{(H1)}, ${\rm \bf (H2)}_4$ hold and that we observe the strictly stationary 
	random field $Y = \{Y_t,\ t\in \Z^d\}$.	 Further assume that either
	\begin{itemize}
		\item[(i)] the field $Y$ is $m$-dependent or
		\item[(ii)] the random field $Y$ is $\phi$-mixing such that equations (\ref{Rosphimix})--(\ref{eq:B}) hold.
	\end{itemize} 
	Then for all $l \in \N$
	\begin{align}\label{EqErrorGenCat}
  \lVert g_1-\hat{g}_{1,l}\rVert_\cdot^2\leq \lVert g_1-g_{1,l} \rVert_2^2+\frac{K}{N}\left(\sqrt{\E|Y_0|^4} + \lVert g_1 \rVert_1^2 \right) \int_{-\pi l}^{\pi l} \frac{dx}{|\psi(x)|^2},
	\end{align}
	where $K>0 $ is a constant, $g_{1,l}$ is given by $g_{1,l}(x) = \frac{1}{2\pi}\int_{-\pi l}^{\pi l} e^{-iux}\frac{\theta(u)}{\psi(u)}du$ for $x\in\R$, and $N\in\N$ is the sample size.
\end{theorem6}
Notice that random fields \eqref{eq:X} are $m$--dependent with $m=\diam (\supp f)$ since a simple function $f$ has a compact support. Introduce the notation $L:=\|  g_1 \|_{H^\beta}^2$.

The following corollary is an immediate consequence of Theorem \ref{TheoGenCat}.
\begin{cor1}\label{cor:ErrEst1}
	If additionally \textbf{(H3)} and \textbf{(H4)} hold then the bound in Theorem \ref{TheoGenCat} can be improved to
	\begin{align*}
	\lVert g_1-\hat{g}_{1,l}\rVert_\cdot^2 \leq \lVert g_1-g_{1,l} \rVert_2^2+\frac{\tilde{K}}{N}\left(1+\sqrt{\E|Y_0|^4} \int_{-\pi l}^{\pi l} \frac{dx}{|\psi(x)|^2}\right),
	\end{align*}
	where $\tilde{K}>0$ is constant.	
\end{cor1}

\begin{cor1}\label{cor:ErrEst2}
	Under the assumptions of Corollary \ref{cor:ErrEst1} it holds
	\begin{equation}\label{eq:Bound_g}
	 \lVert g_1-\hat{g}_{1,l}\rVert_\cdot^2\leq
	 \frac{{L}}{\big(1+(\pi l)^2\big)^\beta}   +\frac{\bar{K}}{N}l \big(1+(\pi l)^2\big)^\beta,	
	 \end{equation}
	where  $\bar{K}=2\pi K c_\psi\left(  \sqrt{\E|Y_0|^4}  +  \|  g_1 \|_1^2 \right)$.	
\end{cor1}

The upper bound \eqref{eq:Bound_g} allows to choose the cut--off parameter $l>0$ optimally by minimizing the right--hand side expression in \eqref{eq:Bound_g} numerically. Choosing $N,l\to +\infty$ such that
$l^{1+2\beta}/N\to 0$ yields the $L^2$--consistency of the estimate $\hat{g}_{1,l}$.


\section{Estimation of the L\'evy Density  $v_0$}\label{sec:es}

In the following Section three different estimation approaches will be discussed. 
The plug-in and the Fourier method are both based on formula (\ref{eq:solution}),
whereas the third one, which uses orthonormal bases (OnB's) in $L^2(\R)$, is totally different
from them. For this reason, the problem will be reformulated in terms of $L^2$--OnB's
there. Nevertheless it turns out that the sufficient conditions for the existence of
a solution do not change essentially.

\subsection{Plug-In Estimator} \label{subsect:PlugIn}

Let $\hat{g}_1^{(h)}$ be an estimator for $g_1^{(h)}=h\cdot v_1$. We now consider a simple plug-in estimator
$\hat{g}_0^{(h)}$ of $g_0^{(h)}=h\cdot g_0$ defined by
\begin{equation}\label{eq:plugin}
\begin{split}
\hat{g}_0^{(h)}(x) & = \frac{|f_1|}{n_1}\frac{h(x)}{h(f_1x)}\hat{g}_1^{(h)}(f_1 x) \\ 
&+ \sum\limits_{j=1}^{n_N} (-1)^j \!\!\!\!\!  \sum\limits_{i_1: f_{i_1}\neq f_1}\!\!\!\!
\dots \!\!\!\! \sum\limits_{i_j: f_{i_j}\neq f_1} \!\!\!\! \frac{\left( |f_1|/n_1\right)^{j+1}}{|f_{i_1}\dots f_{i_j}|} \frac{h(x)}
{h \left( \frac{f_1^{j+1}}{f_{i_1}\dots f_{i_j}}x\right) } 
\hat{g}_1^{(h)}\left( \frac{f_1^{j+1}}{f_{i_1}\dots f_{i_j}}x \right),
\end{split}	
\end{equation}
where $N \in \mathbb{N}$ denotes the sample size and $n_N$ is a certain cut-off parameter 
depending on $N$. The following theorem gives a bound for the mean square error
 $||g_0^{(h)}-\hat{g}_0^{(h)}||_\cdot$.\\ 

\begin{theorem4} \label{thm:PlugIn}
	Consider $g_0^{(h)} \in L^2(\R)$ and let $\hat{g}_1^{(h)} \in L^2(\R)$ be an estimator of $g_1^{(h)}$.
	Let furthermore the conditions of Theorem \ref{th3} be fulfilled. Then 
	with the notation given there it holds
\begin{equation}\label{eq:L2ErrorPlgIn}
\begin{split}
\| g_0^{(h)}&-\hat{g}_0^{(h)} \|_\cdot   \leq \frac{|f_1|^{1/2}  }{n_1} s(f_1) \\
&\times  \left(
\left(1 + \sum\limits_{j=1}^{n_N} (e(f,h))^{j}  \right) \|g_1^{(h)}-\hat{g}_1^{(h)} \|_\cdot  
 +  \frac{  \left( e(f,h) \right)^{n_N + 1} ||g_1^{(h)}||_2 }{1-e(f,h) } \right).
\end{split}
\end{equation}
In particular, if $\hat{g}_1^{(h)}$ is an $L^2$-consistent estimator for $g_1^{(h)}$ (i.e., $\|g_1^{(h)}-\hat{g}_1^{(h)}\|_
{\cdot} \to 0 $  as $N,n_N\to \infty$) then  $\hat{g}_0^{(h)}$ is as well an $L^2$-consistent 
estimator for $g_0^{(h)}$.  
\end{theorem4}

\begin{proof}
First of all, we observe that for each $k \in \mathbb{N}$ and $f_{i_1},\dots,f_{i_k} \neq f_1$ it holds
\begin{equation*}
\begin{split}
\frac{|h(x)|}{|h(\frac{f_1^{k+1}}{f_{i_1}\cdots f_{i_k}} x )|} &=
\frac{|h(x)|}{|h(\frac{f_1}{f_{i_1}}x)|}
\frac{|h(\frac{f_1}{f_{i_1}}x)|}{|h(\frac{f_1^2}{f_{i_1} f_{i_2}}x )|} 
\cdots
\frac{|h(\frac{f_1^{k-1}}{f_{i_1}\cdots f_{i_{k-1}}}x)|}{|h(\frac{f_1^k}{f_{i_1}\cdots f_{i_k}} x )|}
\frac{|h(\frac{f_1^{k}}{f_{i_1}\cdots f_{i_k}}x)|}{|h(\frac{f_1^{k+1}}{f_{i_1}\cdots f_{i_k}} x )|} \\
& \leq s_{i_1} s_{i_2} \cdots s_{i_k} s(f_1).
 \end{split}	
\end{equation*}
By relation \eqref{eq:mainH} and condition \eqref{supcondition}, $g_1^{(h)} \in L^2(\R)$ as well, cf. Lemma \ref{lemm:g_int}.
Using formula (\ref{eq:solution}) it follows by triangle inequality and a simple integral 
substitution that
\begin{equation*}
\begin{split}
&  \|g_0^{(h)}-\hat{g}_0^{(h)} \|_\cdot  \leq \frac{|f_1|^{1/2}}
{n_1} s(f_1) \|g_1^{(h)}-\hat{g}_1^{(h)}\|_\cdot \\
& + \sum\limits_{k=1}^{n_N}\sum\limits_{i_1: f_{i_1}\neq f_1} \dots \sum\limits_{i_k: f_{i_k}\neq f_1} \frac{1}{n_1^{k+1}}
\left( \frac{|f_1|^{k+1}}{|f_{i_1}\cdots f_{i_k}|}\right)^{1/2} s\left(\frac{f_1^{k+1}}{f_{i_1}\cdots f_{i_k}}\right)
\|g_1^{(h)}-\hat{g}_1^{(h)}\|_\cdot \\
& + \sum\limits_{k=n_N+1}^{\infty}\sum\limits_{i_1: f_{i_1}\neq f_1} \dots \sum\limits_{i_k: f_{i_k}\neq f_1} \frac{1}{n_1^{k+1}}
\left( \frac{|f_1|^{k+1}}{|f_{i_1}\cdots f_{i_k}|}\right)^{1/2} s\left(\frac{f_1^{k+1}}
{f_{i_1}\cdots f_{i_k}}\right)
||g_1^{(h)}||_2  \\
&\leq \frac{|f_1|^{1/2}}
{n_1} s(f_1) \|g_1^{(h)}-\hat{g}_1^{(h)}\|_\cdot \\
& + \sum\limits_{k=1}^{n_N}\sum\limits_{i_1: f_{i_1}\neq f_1} \dots \sum\limits_{i_k: f_{i_k}\neq f_1} \frac{1}{n_1^{k+1}}
\left( \frac{|f_1|^{k+1}}{|f_{i_1}\cdots f_{i_k}|}\right)^{1/2} s_{i_1} s_{i_2} \cdots s_{i_k} s(f_1)
\|g_1^{(h)}-\hat{g}_1^{(h)}\|_\cdot \\
& + \sum\limits_{k=n_N+1}^{\infty}\sum\limits_{i_1: f_{i_1}\neq f_1} \dots \sum\limits_{i_k: f_{i_k}\neq f_1} \frac{1}{n_1^{k+1}}
\left( \frac{|f_1|^{k+1}}{|f_{i_1}\cdots f_{i_k}|}\right)^{1/2} s_{i_1} s_{i_2} \cdots s_{i_k} s(f_1)||g_1^{(h)}||_2  \\
& = \frac{|f_1|^{1/2}}
{n_1} s(f_1)\left(
\left(1 + \sum\limits_{j=1}^{n_N}  \left( \frac{1}{n_1}\sum\limits_{k: f_k \neq f_1} \left( \frac{|f_1|}{|f_k|} \right)^{1/2} s_k
\right)^{j} \right) \| g_1^{(h)}-\hat{g}_1^{(h)} \|_\cdot \right.\\
& \left. +  \left( 1-\frac{1}{n_1}\sum_{k: f_k \neq f_1} \left( \frac{|f_1|}{|f_k|} 
\right)^{1/2} s_k \right)^{-1} 
 \left( \frac{1}{n_1}\sum\limits_{k: f_k \neq f_1} \left( \frac{|f_1|}{|f_k|} \right)^{1/2} s_k
\right)^{n_N + 1} \|g_1^{(h)}\|_2\right).
\end{split}	
\end{equation*}
Since $\frac{1}{n_1}\sum_{k: f_k \neq f_1} \left( \frac{|f_1|}{|f_k|} \right)^{1/2} s_k < 1$ the consistency 
result follows immediately from this approximation.
\end{proof}

\begin{le2}\label{lemm:g_int}
	 Let $g_0^{(h)}\in L^p(\R)$, $p\ge 1$, and condition \eqref{supcondition} hold. Then $g_1^{(h)}\in L^p(\R)$.
\end{le2}
\begin{proof}
Using relation \eqref{eq:mainH}, condition \eqref{supcondition} and triangle inequality, we get 
$$
\| g_1^{(h)} \|_p \le \sum\limits_{k=1}^n s(1/f_k) \| g_0^{(h) }\|_p .
$$
\end{proof}

Using the estimator $\hat{g}^{(h)}_0$ in practice reveals that 
\begin{enumerate}
\item the choice $n_N=1,2,3$ suffcies completely to get good results due to fast convergence of the geometric series in \eqref{eq:L2ErrorPlgIn},
\item $\hat{g}^{(h)}_0$ oscillates much in a neighborhood of the origin.
\end{enumerate}
Hence, one has to regularize it applying a usual smoothing procedure. Convolve $\hat{g}^{(h)}_0$ with a smoothing kernel $K_b: \R\to\R_+$ which depends on its bandwidth $b>0$ and satisfies the following assumptions:

\begin{description}
	\item[(K1)] $K_b\in L^1(\R)\cap L^2(\R)$, $\int_\R K_b(x)\, dx=1$ for all $b>0$
	\item[(K2)] $\sup_x |\F [K_b](x) |   \le  C_K$ where $C_K\in (0,+\infty)$ is a constant independent of $b>0$
	\item[(K3)] $|1-\F [K_b](x)|  \le c_1 \min\{1, b |x|\}  $ for all $b>0$, $x \in \R$ where $c_1>0$ is a constant.	
\end{description}

For the resulting estimator
\begin{equation}\label{eq:SmoothPlugIn}
\tilde{g}^{(h)}_0(x)= \hat{g}^{(h)}_0 * K_b (x)=\int\limits_\R K_b(x-y)  \hat{g}^{(h)}_0(y) \, dy
\end{equation}
we give an upper bound of its mean square error and prove its consistency as $N, n_N\to \infty$ and $b\to +0$.

\begin{theorem4}\label{thm:ErrorSmoothedPI}
		Let $g^{(h)}_0 \in L^1(\R) \cap H^{\delta}(\R)$ for some $\delta > 1/2$, and let  $\hat{g}^{(h)}_1 \in L^1(\R) \cap L^2(\R)$ be an estimator of 
		${g}^{(h)}_1$. For a kernel $K_b$ satisfying assumptions {\rm { \bf (K1)} --{\bf  (K3)}}, $b\in (0,1)$ it holds
\begin{equation} \label{eq:ErrSmoothedPI}
\|g_0^{(h)}-\tilde{g}^{(h)}_0||_\cdot \leq \frac{C_K}{2\pi}  \|g_0^{(h)}-\hat{g}_0^{(h)}\|_\cdot 
+   ||g_0^{(h)}||_{1}^{1/2}  ||g_0^{(h)}||_{H^\delta}^{1/2} a_\delta(b), 
\end{equation}
where 
\begin{equation}\label{eq:a_delta}
	a_\delta(b) = \begin{cases}
	 \mathcal{O}\left( b^{1 \wedge \frac{(2\delta-1)}{4}} \right), & \ \delta \neq 5/2, \\
	 \mathcal{O}\left( b (- \log b)^{1/4} \right), & \ \delta = 5/2.
	\end{cases}
	\end{equation}
\end{theorem4}

\begin{proof}
By triangle inequality, Plancherel identity and convolution property of $\F$ we have
\begin{equation*}
\begin{split}
\|g_0^{(h)}-\tilde{g}^{(h)}_0||_\cdot & \leq  \|g_0^{(h)}*K_b-\hat{g}_0^{(h)}*K_b\|_\cdot  + \|g_0^{(h)}*K_b-{g}_0^{(h)}\|_\cdot \\
& \leq  \frac{1}{2\pi}\| \F[g_0^{(h)}-\hat{g}_0^{(h)}]\F[K_b]\|_\cdot  + \frac{1}{2\pi} \| \F[g_0^{(h)} ] (\F[ K_b]-1 ) \|_2 \\
& \leq \frac{C_K}{2\pi}  \|g_0^{(h)}-\hat{g}_0^{(h)}\|_\cdot 
+   \frac{1}{2\pi}  \|\F[ g_0^{(h)} ]  (  \F[K_b]-1) \|_2,
\end{split}	
\end{equation*}
since $\hat{g}^{(h)}_0 \in L^1(\R) \cap L^{2}(\R)$ by relation \eqref{eq:plugin}. 
By assumption {\bf (K3)} and Cauchy-Schwartz inequality, we have
	\begin{align*}
	||\F[g_0^{(h)}]&(1 - \F[K_b])||_{2} 
	 = \left( \int_{\R} |\F [g_0^{(h)}](x)|^2 |1-\F[K_b] (x)|^2 dx \right)^{1/2} \\
	& {\leq} c_1 ||g_0^{(h)}||_{1}^{1/2} \left( \int_{\R} |\F [g_0^{(h)}](x)|(1+x^2)^{\delta/2} (1+x^2)^{-\delta/2} \min\{ 1,b|x| \}^2 dx \right)^{1/2} \\
	& {\leq}c_1  ||g_0^{(h)}||_{1}^{1/2} ||g_0^{(h)}||_{H^\delta}^{1/2} \left( \int_{\R} \min\{ 1,b|x| \}^4 (1+x^2)^{-\delta}   dx \right)^{1/4}.
	\end{align*}
	The rest of the proof follows by observing that for $b\in(0,1)$
	\begin{align*}
	a_\delta(b)&:= \frac{c_1}{2\pi} \int_{\R} \min\{ 1,b|x| \}^4 (1+x^2)^{-\delta} dx \\
	&\le \begin{cases}
	\frac{c_1}{\pi} b^4 \left(  \int\limits_0^1 \frac{x^4 \, dx}{(1+x^2)^\delta}+\frac{1}{5-2\delta}+ \frac{b^{2\delta-5}}{2\delta-5} \right) + \frac{c_1}{\pi(2\delta-1)}b^{2\delta-1}=
	\mathcal{O}\left( b^{4 \wedge (2\delta-1)} \right), & \ \delta \neq 5/2, \\
	\frac{c_1}{\pi} b^4 \left(   \int\limits_0^1 \frac{x^4 \, dx}{(1+x^2)^\delta}+\frac{1}{4}\right) -\frac{c_1}{\pi}b^4 \log b   =   \mathcal{O}\left( -b^4 \log b \right), &\ \delta = 5/2, 
	\end{cases}\end{align*}
	as $h \rightarrow 0$.		
\end{proof}
There are many examples of kernels satisfying assumptions {\bf (K1)}--{\bf (K3)}, e.g., the Gaussian kernel $K_b(x)=\frac{1}{\sqrt{2\pi} b} e^{-x^2/(2b^2)}$. Since $\F[K_b](x)=e^{-b^2 x^2/2}$, {\bf (K1)}--{\bf (K2)} are trivial. Condition {\bf (K3)} holds from the  inequality $$ |\F[ K_b]-1|=| e^{-b^2x^2/2} - 1 |\le b^2x^2/2\le 2\min\{  1, b|x| \},\quad x\in\R,\; b>0.$$ 

Another class of examples is provided by $K_b(x) = K(x/b)/b$, $x \in \R$, where $K \in L^1(\R) \cap L^2(\R)$  is a nonnegative function such that $\int_\R K(x)\, dx=1$, $\supp (\F[K]) \subseteq [-1,1]$ and $\F [K]$ is a Lipschitz continuous function. While {\bf (K1)}--{\bf (K2)} trivially hold in this case, {\bf (K3)}  can be seen from the following lemma.
	
\vspace{0.5cm}
\begin{le1}\label{bw:lem1}
	Let $K:\R \rightarrow \R_+$ be as above. Then 
	$$ |1-\F[K_b](x)| \leq c_1 \min \{ 1,b |x| \}, \quad x \in \R, $$
	where $c_1 = \max \{ 1, L_K \}$ with $L_K > 0$ being the Lipschitz constant of $K$.
\end{le1}
\begin{proof}
	Because of $\supp (\F[K]) \subseteq [-1,1]$ and the Lipschitz continuity of $\F[K]$ it follows
	\begin{equation*}
	|1-\F[K_b](x)| = |1-\F[K](bx)| = \begin{cases}
	1 ,& \ b |x| > 1, \\
	\leq L_K b |x|, & \ b |x| \leq 1.
	\end{cases}
	\end{equation*}
	Thus $|1-\F[K_b](x)| \leq c_1 \min \{ 1,b|x| \}$.
\end{proof}
\begin{cor4}\label{cor:consist_smoothed}
Choose $\hat{g}^{(h)}_1=\hat{g}_{1,l}$, $h(x)=x$ as in Section \ref{sec:est_g_1}.  Under the assumptions of Theorem \ref{th3}, Corollary \ref{cor:ErrEst1} and Theorem \ref{thm:ErrorSmoothedPI}  the estimator
$\tilde{g}_0^{(h)}$ is  $L^2$-consistent 
 for $g_0$ as $N, n_N\to \infty$ and $b\to +0$. 
\end{cor4}
\begin{proof}
Applying Theorem \ref{thm:PlugIn} and Corollary \ref{cor:ErrEst2} yields
$\|g_0-\hat{g}^{(h)}_0||_\cdot\to 0$ as $N, l\to \infty$ for any sequence $n_N\to \infty$. Relation $a_{\delta}(b)\to 0 $ as $b\to +0$ finishes the proof.	
\end{proof}
\begin{rem4}\label{rem:choice_bandwidth}
The choice of bandwidth $b>0$ in \eqref{eq:SmoothPlugIn} can be made by solving the following minimization problem numerically:
$$
\left\| \frac{\partial \tilde{g}^{(h)}_0}{\partial b}  \right\|_2 \to \min_{b>0},
$$
which means that we are seeking for a sufficiently smooth estimate $\tilde{g}^{(h)}_0$. Assuming that $K_b$ is a $C^1$--smooth function of parameter $b>0$ and that the differentiation with respect to $b$ and the integral can be interchanged we get by Plancherel identity and convolution property of $\F$ that
$$
\left\| \frac{\partial \tilde{g}^{(h)}_0}{\partial b}  \right\|_2=
\left\| \F \left[   \frac{\partial \tilde{g}^{(h)}_0}{\partial b} \right]    \right\|_2=
\left\| \F \left[    \hat{g}^{(h)}_0 \right]   \F \left[   \frac{\partial K_b}{\partial b} \right]   \right\|_2
 \to \min_{b>0}.
$$
For easy particular functions $K_b$ the Fourier transform of $\frac{\partial K_b}{\partial b} $ can be usually calculated explicitly. In contrast, $\F \left[    \hat{g}^{(h)}_0 \right] $ has to be estimated from the data, compare Section \ref{subsect:Fourier} for $h(x)=x$. There, we use the estimate $\widehat{\F [g_0]}$ to assess $\F \left[    \hat{g}_0 \right] $.
\end{rem4}


\subsection{Fourier Approach} \label{subsect:Fourier}

A common strategy in the estimation of $g_1^{(h)}$ (e.g. in the case of L\'{e}vy processes) is first
to estimate its Fourier transform $\mathcal{F}[g_1^{(h)}] $ and then to invert it. This causes an error in the estimation of the Fourier transform and additionally
in the inversion procedure. Using plug-in estimators  of Section \ref{subsect:PlugIn}, this may increase the estimation error for $g_0^{(h)}$. For this reason, here we estimate $\mathcal{F}[g_1^{(h)}] $ directly to recover $g_0^{(h)}$. 

From now on, set  $h(x) = x^\beta$ for some $\beta \in \mathbb{N}$. In other words, equation (\ref{eq:solution}) is of the form 
\begin{equation}\label{eq:solutionF}
\begin{split}
& g_0(\cdot) =  \frac{1}{n_1}\sign(f_1)^\beta |f_1|^{1-\beta} g_1(f_1 \cdot) \\
& + \sum\limits_{j=1}^\infty (-1)^j \!\!\!\! \!\! \sum\limits_{i_1: f_{i_1}\neq f_1}\!\!\!\!
\dots \!\!\!\! \sum\limits_{i_j: f_{i_j}\neq f_1} \!\! \frac{1}{n_1^{j+1}} \sign \! \left( \frac{f_1^{j+1}}{f_{i_1}\dots f_{i_j}} \right)^\beta  \!\!
\left( \frac{|f_1|^{j+1}}{|f_{i_1}\dots f_{i_j}|} \right)^{1-\beta} \!\!\!\!\!\!\!
g_1\left( \frac{f_1^{j+1}}{f_{i_1}\dots f_{i_j}}\cdot \right)\!\!,
\end{split}
\end{equation} 
where $g_0(x) = x^\beta v_0(x)$ and $g_1(x) = x^\beta v_1(x)$. Suppose that $g_0\in L^1(\R)$ and the conditions
of Theorem \ref{th3} are fulfilled. Then $g_1\in L^1(\R)$ as well by Lemma \ref{lemm:g_int}.

The following construction of  $\hat{g}_{0,l}(t)$ and $\hat{g}_{1,l}(t)$ is motivated by estimation
approaches for the characteristic triplet of L\'{e}vy processes (see e.g. \cite{comte}).
Taking Fourier transforms on both sides of (\ref{eq:solutionF}) yields
\begin{equation*}
\begin{split}
& \F[g_0](t) = \frac{1}{n_1}\sign(f_1)^\beta |f_1|^{-\beta} \F\left[ g_1 \right] \left( \frac{t}{f_1} \right)  \\ 
&+ \sum\limits_{j=1}^\infty (-1)^j \sum\limits_{i_1: f_{i_1}\neq f_1} 
\dots \sum\limits_{i_j: f_{i_j}\neq f_1} \frac{1}{n_1^{j+1}} \sign \left( \frac{f_1^{j+1}}{f_{i_1}\dots f_{i_j}} \right)^\beta 
\left( \frac{|f_1|^{j+1}}{|f_{i_1}\dots f_{i_j}|} \right)^{-\beta}
\F\left[ g_1 \right] \left( \frac{f_{i_1}\dots f_{i_j}}{f_1^{j+1}} t \right)
\end{split}	
\end{equation*} 
for $t \in \R$. Let $\widehat{\F [g_1]}$ be any estimator for the Fourier transform of $g_1$. Then we define
the estimator $\widehat{\F [g_0]}$ for $\F [g_0]$ via
\begin{equation*}
\begin{split}
& \widehat{\F [g_0]}(t) = \frac{1}{n_1}\sign(f_1)^\beta |f_1|^{-\beta} \widehat{\F [g_1]} \left( \frac{t}{f_1} \right)  \\ 
&+ \sum\limits_{j=1}^{n_N} (-1)^j \sum\limits_{i_1: f_{i_1}\neq f_1}
\dots \sum\limits_{i_j: f_{i_j}\neq f_1} \frac{1}{n_1^{j+1}} \sign \left( \frac{f_1^{j+1}}{f_{i_1}\dots f_{i_j}} \right)^\beta 
\left( \frac{|f_1|^{j+1}}{|f_{i_1}\dots f_{i_j}|} \right)^{-\beta}
\widehat{\F [g_1]} \left( \frac{f_{i_1}\dots f_{i_j}}{f_1^{j+1}} t \right),
\end{split}	
\end{equation*}
$t \in \R$. If $\widehat{\F[g_1]}$ is locally square integrable, an estimator $\hat{g}_{0,l}$ of $g_0$ 
is constructed for some $l > 0$ as
\begin{equation}\label{est:Fourier}
 \hat{g}_{0,l}(t) = \frac{1}{2 \pi} \int\limits_{-\pi l}^{\pi l} e^{-itu} \widehat{\F [g_0]}(u)du,
 \quad t \in \R.
\end{equation}
The last expression can be rewritten as
\begin{equation*}
\begin{split}
 & \hat{g}_{0,l}(t)  = \frac{1}{n_1}\sign(f_1)^\beta |f_1|^{-\beta} \frac{1}{2 \pi} \int\limits_{-\pi l}^{\pi l} e^{-itu} 
  \widehat{\F [g_1]} \left( \frac{u}{f_1} \right)du \\
 & + \sum\limits_{j=1}^{n_N} (-1)^j \sum\limits_{i_1: f_{i_1}\neq f_1}
\dots \sum\limits_{i_j: f_{i_j}\neq f_1} \frac{1}{n_1^{j+1}} \sign \left( \frac{f_1^{j+1}}{f_{i_1}\dots f_{i_j}} \right)^\beta 
\left( \frac{|f_1|^{j+1}}{|f_{i_1}\dots f_{i_j}|} \right)^{-\beta} \\
& \times \frac{1}{2 \pi} \int\limits_{-\pi l}^{\pi l} 
e^{-itu} \widehat{\F [g_1]} \left( \frac{f_{i_1}\dots f_{i_j}}{f_1^{j+1}} u \right) du \\
&= \frac{1}{n_1}\sign(f_1)^\beta |f_1|^{1-\beta} \hat{g}_{1,\frac{l}{|f_1|}}(f_1 t)  \\
& + \sum\limits_{j=1}^{n_N} \frac{(-1)^j}{n_1^{j+1}} \sum\limits_{i_1: f_{i_1}\neq f_1}
\dots \sum\limits_{i_j: f_{i_j}\neq f_1} \sign \left( \frac{f_1^{j+1}}{f_{i_1}\dots f_{i_j}} \right)^\beta 
\left( \frac{|f_1|^{j+1}}{|f_{i_1}\dots f_{i_j}|} \right)^{1-\beta} \\
&\times
\hat{g}_{1, \left| \frac{f_{i_1}\dots f_{i_j}}{f_1^{j+1}} \right|  }
\left( \frac{f_1^{j+1}}{f_{i_1}\dots f_{i_j}} t \right)
\end{split} 
\end{equation*}
with $\hat{g}_{1,l}(t) = \frac{1}{2 \pi} \int_{-\pi l}^{\pi l} e^{-itu} \widehat{\F [g_1]}(u)du$ being
an estimator of $g_1$. 
\begin{rem4}\label{rem:new}
The estimator (\ref{eq:Fg_1_est}) from Section \ref{sec:est_g_1} is locally square integrable.
	In this case an appropriate choice for the parameter $l > 0$ can be achieved
	e.g. by minimizing the right-hand side of (\ref{eq:Bound_g}) for any fixed sample size $N$ (see also the
	discussion following Corollary \ref{cor:ErrEst2}).
\end{rem4}

Similar as in Theorem \ref{thm:PlugIn} one can obtain an upper bound for the $L^2$-error. With the notation 
$g_{1,l}(t) = \frac{1}{2 \pi} \int_{-\pi l}^{\pi l} e^{-itu} \F [g_1](u)du$ we get
\begin{equation}\label{eq:estFourier}
\begin{split}
&  \Vert \hat{g}_{0,l} - g_0 \Vert_\cdot  \leq 
\frac{1}{n_1|f_1|^\beta} \left(  \left\Vert \hat{g}_{1,\frac{l}{|f_1|}} - g_1\right
\Vert_\cdot +  \frac{\left( \frac{1}{n_1}\sum\limits_{k: f_k\neq f_1} s_k \right)^{n_N + 1} }{1 - \frac{1}{n_1}\sum\limits_{k: f_k\neq f_1} s_k} \Vert g_1 \Vert_2 \right.\\
& \left.+ \sum\limits_{j=1}^{n_N} \sum\limits_{i_1: f_{i_1}\neq f_1}
\dots \sum\limits_{i_j: f_{i_j}\neq f_1} \frac{s_{i_1} \ldots s_{i_j}   }{n_1^{j}} \left\Vert \hat{g}_{1, \left| \frac{f_{i_1}\dots f_{i_j}}{f_1^{j+1}} \right| l} - 
g_1 
\right\Vert_\cdot \right) ,
\end{split}
\end{equation}
where $s_k=\left(  |f_k|/|f_1| \right)^\beta$, $k=2,\ldots,n$. Assume
\begin{equation}\label{eq:cond_s_k}
e(f, |\cdot|^{\beta+1/2})=\frac{1}{n_1}\sum\limits_{k: f_k\neq f_1} s_k < 1.
\end{equation}
Choose the estimator $\hat{g}_{1,l}$ of $g_{1}$ in an $L^2$--consistent way. Then, as
$N,l, n_N \to \infty$  in an appropriate manner, the above upper bound \eqref{eq:estFourier} tends to zero,
and $\hat{g}_{0,l}$ is $L^2$--consistent for $g_0$. For instance, one can choose $\hat{g}_{1,l}$ from Section \ref{sec:est_g_1}, which is $L^2$--consistent under assumptions of Corollary \ref{cor:ErrEst1}.

Assume, in addition to \eqref{eq:cond_s_k}, that $|f_1|>\max_{k: f_k\neq f_1} |f_k| .$ By \eqref{eq:Bound_g}, the upper bound of $\| g_1- \hat{g}_{1,l}\|_\cdot$ is monotonously non--decreasing in $l$. Since 
$$
\left|  \frac{f_{i_1}\ldots f_{i_j}}{f_{1}^{j+1}}  \right| l< \frac{l}{|f_1|} 
$$
we get by \eqref{eq:estFourier} and \eqref{eq:Bound_g} that
\begin{equation*}\label{eq:estFourier1}
\begin{split}
 \Vert \hat{g}_{0,l} - g_0 \Vert_\cdot  & \leq 
\frac{1}{n_1|f_1|^\beta} \left(  \left( 1+  \sum_{j=1}^{n_N} e^j(f, |\cdot|^{\beta+1/2})  \right) {\cal O} \left(  \frac{1}{l^{2\beta}}+\frac{l^{2\beta+1}}{N} \right) \right.\\
 &\left. +  \frac{\left( e(f, |\cdot|^{\beta+1/2}) \right)^{n_N + 1} }{1 - e(f, |\cdot|^{\beta+1/2})} \Vert g_1 \Vert_2 \right)\\
&\le \frac{1}{n_1|f_1|^\beta} \left(  \left( 1+  \frac{e(f, |\cdot|^{\beta+1/2}) }{1 - e(f, |\cdot|^{\beta+1/2})}    \right)  {\cal O} \left(  \frac{1}{l^{2\beta}}+\frac{l^{2\beta+1}}{N} \right) \right.\\
&\left.+  \frac{\left( e(f, |\cdot|^{\beta+1/2}) \right)^{n_N + 1} }{1 - e(f, |\cdot|^{\beta+1/2})} \Vert g_1 \Vert_2 \right)   \to 0 
\end{split}
\end{equation*}
as $N,n_N,l\to \infty$ such that $\frac{l^{2\beta+1}}{N}\to 0$.

\subsection{Orthonormal Basis Approach}\label{subsect:OnB}

Since the series representation (\ref{eq:solution}) is sensitive to noise
and bad estimates for $v_1$, the aim is to obtain an estimation approach that uses (local)
orthonormal bases (e.g., Haar wavelets) of $L^2$. Moreover, from the numerical point
of view it is much more convenient to find a solution only on a finite interval.
For this reason, the problem of Section \ref{sect:InvProb} should be reformulated for functions on
$L^2(\R)$ with support contained in a finite interval. For
 $0 < A < \infty$, consider  
$$ U_A = \{u \in L^2(\R): \ u = 0 \ \text{a.e. on} \ \R \backslash [-A,A]\} $$
to be  the closed linear subspace of $L^2(\R)$ equipped with
the usual scalar product on $L^2(\R)$. Find
a function $g_0^{(h)} \in U_A$ that fulfills equation (\ref{eq:mainH}) for fixed $g_1^{(h)}$.
Because of the scalings on the right hand side of this equation, we have to extend the 
assumptions on $g_1^{(h)}$ and the coefficients $|f_j|$ a bit. Let $|f_1| \geq \max_{k: f_k \neq f_1} |f_k|$
be the largest coefficient and define $M = \min\{1, |f_1|\}$. Then for $g_1^{(h)} \in U_{AM}$ 
it follows that $g_1^{(h)}(f_1 \cdot) \in U_A$. 
Since $|f_1|$ is the largest coefficient, it holds moreover that $g_0^{(h)}(f_1/f_j x) = 0$,
for all $|x| > A$, i.e. $g_0^{(h)}(f_1/f_j \cdot) \in U_A$ for all $j=1,\dots,n$.
For this reason the restriciton $\varphi_{g_1^{(h)}}|_{U_A}$ of the function $\varphi_{g_1^{(h)}}$ from the proof 
of Theorem \ref{th3} is a map on $U_A$. Then one can show the following theorem with the
same arguments applied to $\varphi_{g_1^{(h)}}|_{U_A}$.  

\begin{theorem2}\label{th:2}
Let $h:\R \rightarrow \R$, $s_k$ be as in Theorem \ref{th3}, and let $g_1^{(h)} \in U_{AM}$ with
$M$ defined as before. Assume furthermore that $|f_1| \geq \max_{k: f_k \neq f_1} |f_k|$ and
relation \eqref{eq:mcond} holds.
Then there exists a unique function $g_0^{(h)} \in U_A$ such that 
\begin{equation}\label{eq:A}
 g_1^{(h)}(\cdot) = \sum\limits_{k=1}^n \frac{1}{|f_k|}\frac{h(\cdot)}{h(\frac{\cdot}{f_k})}g_0^{(h)}(\frac{\cdot}{f_k})
\end{equation}
a.e. on $[-AM, AM]$. The solution $g_0^{(h)}$ can be expressed as in (\ref{eq:solution}).
\end{theorem2}
Note that the solution $g_0^{(h)}$ fulfills the equation $\varphi_{g_1^{(h)}}|_{U_A}(g_0^{(h)}) = g_0^{(h)}$ a.e. on the whole
interval $[-A,A]$, whereas (\ref{eq:A}) holds only on $[-AM, AM]$, which  is merely the same
if $M = 1$, i.e. in the case $|f_1| \geq 1$. Notice that $g_1^{(h)} \in {U_{AM}}$ means that the random field
$X$ has a compound Poisson marginal distribution if $h(x)\equiv 1$.

The last theorem stated the existence of a solution
$g_0^{(h)}$ of the fixpoint equation $\varphi_{g_1^{(h)}}|_{U_A}(g_0^{(h)}) = g_0^{(h)}$ or equivalently for
\begin{equation}\label{eq:helpme} 
\bar{g}_1(\cdot):=\frac{h(\cdot)}{h(f_1 \cdot)}g_1^{(h)}(f_1 \cdot) = \sum\limits_{k=1}^n \frac{1}{|f_k|}
\frac{h(\cdot)}{h\left( \frac{f_1}{f_k} \cdot \right) } g_0^{(h)}\left(\frac{f_1}{f_k} \cdot \right).  
\end{equation}
Now let $(\psi_n)_{n \in \N}$ be an orthonormal basis (OnB) of $U_A$. Since $g_0^{(h)} = \sum_{j=1}^\infty 
\left\langle g_0^{(h)}, \psi_j \right\rangle \psi_j $ it holds
\begin{equation}\label{eq:fp} 
\sum\limits_{k=1}^n \frac{1}{|f_k|}
\frac{h(\cdot)}{h\left( \frac{f_1}{f_k} \cdot \right) } g_0^{(h)}\left(\frac{f_1}{f_k} \cdot \right)
= \sum\limits_{j=1}^\infty 
\left\langle g_0^{(h)}, \psi_j \right\rangle \sum\limits_{k=1}^n \frac{1}{|f_k|}
\frac{h(\cdot)}{h\left( \frac{f_1}{f_k} \cdot \right) } \psi_j\left(\frac{f_1}{f_k} \cdot \right).
\end{equation}
Note that because of $|f_1| \geq \max_{k: f_k \neq f_1} |f_k|$ the function $\psi_j(f_1/f_k \cdot)$ is
in $U_A$ for all $k \in \mathbb{N}$. Set
$$ \eta_j (\cdot) = \sum\limits_{k=1}^n \frac{1}{|f_k|}
\frac{h(\cdot)}{h\left( \frac{f_1}{f_k} \cdot \right) } \psi_j\left(\frac{f_1}{f_k} \cdot \right),
\quad j \in \mathbb{N} .$$
Then we can conclude that there exists a solution $g_0^{(h)} \in U_A$ of (\ref{eq:helpme}) if and only if
the function $\bar{g}_1$ admits a representation $\bar{g}_1 = \sum_{j=1}^\infty x_j \eta_j$ with some
$l^2$-sequence $(x_j)_{j \in \mathbb{N}}$. In this case,  a solution $g_0^{(h)} $ is given by 
$\sum_{j=1}^\infty x_j \psi_j$. It is unique if and only if the scalar sequence $(x_j)_{j \in \mathbb{N}}$
is unique. In other words, the problem is characterized by the operator $T:l^2 \rightarrow U_A$,
$$ T: x = (x_j)_{j \in \mathbb{N}} \mapsto \sum_{j=1}^\infty x_j \eta_j. $$
If $T$ is surjective there exists a solution. If it is bijective the solution is unique. 
It is clear now that under the conditions of Theorem \ref{th:2} the operator $T$
is a bijection. Nevertheless, let us reformulate this theorem in terms of the 
OnB $(\psi_l)_{l \in \mathbb{N}}$ and give another proof for it.\\

\begin{theorem5}\label{theoSolSec}
Let $(\psi_l)_{l \in \mathbb{N}}$ be an OnB of $U_A$, and let the conditions of Theorem \ref{th:2}
be fulfilled. Then there exists a unique sequence $x \in l^2$ such that the operator $T$ is one--to--one.
\end{theorem5}

\begin{proof}
We would like to show that the system $(\eta_j)_{j\in \mathbb{N}}$ is a basis for $U_A$.
First we show, by contradiction, that 
$$V:=\overline{\text{span}}((\eta_j)_{j\in \mathbb{N}}) = U_A.$$
Therefore assume that $V \subset U_A$. Since $V$ is a closed subspace of $U_A$
it follows by Riesz lemma (see e.g. \cite{Werner11}) that for any $0<\delta < 1$ there exists
a function $g_\delta \in U_A$ with $\Vert g_\delta \Vert_2=1$ 
such that $\Vert g_\delta - v \Vert_2 \geq 1-\delta$, for all $v \in V$. Now choose 
$ \delta :=  \left( 1 -e(f,h)  \right)/2  . $
Then we can 
write $g_\delta = \sum_{k=1}^\infty \left\langle g_\delta, \psi_k \right\rangle \psi_k $.
Define the sequence $x = (x_k)_{k \in \mathbb{N}} \in l^2$ via $x_k = |f_1|\left\langle 
g_\delta, \psi_k \right\rangle /n_1$, $k\in \mathbb{N}$. Since $\Vert g_\delta \Vert_2=1$
it follows $\Vert x \Vert_{2}=|f_1|/n_1$. Clearly, we have $\sum_{k=1}^\infty x_k \eta_k \in V$.
By triangle inequality, a substitution in the integral and the definition of $s_k$ it can be observed that
$$ 1 - \delta \leq \left\Vert g_\delta - \sum\limits_{k=1}^\infty x_k \eta_k \right\Vert_2= \left\Vert \frac{1}{n_1}\sum_{j: f_j \neq f_1}\frac{|f_1|}{|f_j|} \frac{h(\cdot)}{h\left( \frac{f_1}{f_j} \cdot \right) } g_\delta  \left(  \frac{f_1}{f_j} \cdot \right) \right\Vert_2 
\leq e(f,h), $$
which is a contradiction to the fact that $1-\delta >e(f,h)$, i.e. $V = U_A$. 

In the second step of the proof, we  use \cite[ Theorem 3.1.4]{Christensen03}
to show that $(\eta_j)_{j\in \mathbb{N}}$ is a basis for $U_A$. 
Therefore we have to verify the assumptions there. 
First of all, we observe that $\eta_l$ are non-zero functions, since
\begin{equation*}
\begin{split}
 \Vert\eta_j \Vert_2 & = \left\Vert \frac{n_1}{|f_1|}\psi_j 
 + \sum\limits_{k: f_k\neq f_1} \frac{1}{|f_k|}
\frac{h(\cdot)}{h\left( \frac{f_1}{f_k} \cdot \right) } \psi_j\left(\frac{f_1}{f_k} \cdot \right) 
\right\Vert_2  \geq \frac{n_1}{|f_1|}\\
& - \left\Vert \sum\limits_{k: f_k\neq f_1} \frac{1}{|f_k|}
\frac{h(\cdot)}{h\left( \frac{f_1}{f_k} \cdot \right) } \psi_j\left(\frac{f_1}{f_k} \cdot \right) 
\right\Vert_2  \geq \frac{n_1}{|f_1|} - \sum\limits_{k: f_k\neq f_1} \frac{1}{|f_k|} \left\Vert  
\frac{h(\cdot)}{h\left( \frac{f_1}{f_k} \cdot \right) } \psi_j\left(\frac{f_1}{f_k} \cdot \right) 
\right\Vert_2 \\
& \geq \frac{n_1}{|f_1|} - \sum\limits_{k: f_k\neq f_1} s_k \left( \frac{1}{|f_k|\cdot |f_1|} \right)^{1/2} 
= \frac{n_1}{|f_1|}\left(  1 - e(f,h)\right), 
\end{split}	
\end{equation*}
where the latter is strictly positive, i.e. $(\eta_j)_{j \in \N}$ is a sequence of non-zero functions 
in the Hilbert space $U_A$. Now let $(c_j)_{j \in \N}$ be an arbitrary real valued sequence and 
$m,l \in \N$ with $m \leq l$. Show that there exists a constant $K$ such that 
$\Vert \sum_{j=1}^m c_j \eta_j \Vert_2 \leq 
K \Vert \sum_{j=1}^l c_j \eta_j \Vert_2$. If $c_1 = c_2 = \dots = c_l = 0$ then
this relation is obviously  true for any choice of $K$. Otherwise,
$$ \left\|  \sum\limits_{j=1}^l c_j \eta_j \right\|_2 \geq 
\frac{n_1}{|f_1|} \left( 1- e(f,h) \right)
\left( c_1^2+\dots+c_l^2 \right)^{1/2} > 0.  $$ Thus, we have
\begin{equation*} 
\begin{split}   
\frac{\big\Vert \sum\limits_{j=1}^m c_j \eta_j \big\Vert_2}
{\big\Vert \sum\limits_{j=1}^l c_j \eta_j \big\Vert_2}
& \leq \frac{e(f,h)\left( c_1^2+\dots+c_m^2 \right)^{1/2}}
{\left( 1-e(f,h) \right)
\left( c_1^2+\dots+c_l^2 \right)^{1/2}}  
 \leq \frac{e(f,h)}{1-e(f,h) }  =: K.
\end{split}
\end{equation*}
This means $(\eta_j)_{j\in \N}$ is a basis for $U_A$, 
i.e. for any function $f \in U_A$
there is a unique scalar sequence $(c_j(f))_{j \in \N}$ with $f = \sum_{j=1}^\infty c_j(f) \eta_j$.
Since 
$$ \Vert f \Vert_2 = \left\Vert \sum\limits_{j=1}^\infty c_j(f) \eta_j \right\Vert_2 \geq 
\frac{n_1}{|f_1|} \left( 1- e(f,h) \right)
\left( \sum\limits_{j=1}^\infty c_j^2(f) \right)^{1/2},  $$
the sequence $(c_j(f))_{j \in \N}$ is furthermore an element of $l^2$. Choosing 
$$ f = \frac{h(\cdot)}{h(f_1 \cdot)}g_1^{(h)}(f_1 \cdot) $$
completes the proof.
\end{proof}    
Note that the proof of the last theorem shows that the system $(\eta_j)_{j \in \mathbb{N}}$ is a
basis for the $L^2$-subspace $U_A$. Therefore we can orthonormalize it by Gram--Schmidt method to an
OnB $(e_j)_{j \in \mathbb{N}}$ of $U_A$ given by $e_1 = \eta_1 / ||\eta_1||_2$ and succesively 
$$e_{k} = \frac{\eta_k-\sum_{i=1}^{k-1}\left\langle \eta_k,e_i \right\rangle e_i }
{\Vert \eta_k-\sum_{i=1}^{k-1}\left\langle \eta_k,e_i \right\rangle e_i \Vert_2}, \quad k=2,3,\dots .$$
Now let $\hat{\bar{g}}_1$ be any estimator for $\bar{g}_1 \in U_A$ and
let $P_m$ be the orthogonal projection of $U_A$ onto the $m$-dimensional subspace
$V_m = \text{span} \{\eta_1,\dots,\eta_m\}=\text{span} \{e_1,\dots,e_m\}$ which is given by $P_m f = \sum_{j=1}^m 
\left\langle f,e_j \right\rangle e_j $. Define the sequence $(\hat{y}_j)_{j \in \mathbb{N}}$ by
$$ \hat{y}_j = \begin{cases}
					\left\langle \hat{\bar{g}}_1, e_j \right\rangle, & \ 1 \leq j \leq m, \\
					0, &  j > m. 
			   \end{cases} $$
Then the orthogonal projection of $\hat{\bar{g}}_1$ onto $V_m$ is
$$ \hat{\bar{g}}_{1,m} := P_m \hat{\bar{g}}_1 = \sum\limits_{j=1}^\infty \hat{y}_j e_j \ \ \left( = \sum\limits_{j=1}^m \hat{y}_j e_j \right) . $$
Now, an estimator $\hat{g}_{0,m}^{(h)}$ for $g_0^{(h)}$ will be constructed as follows:
\begin{enumerate}[1.)]
\item Let $(\hat{x}_{1,m},\dots,\hat{x}_{m,m})$ be the unique solution to
\begin{equation}\label{eq:helpmemore}
\hat{y}_j = \sum\limits_{i=1}^m \hat{x}_{i,m} \left\langle \eta_i, e_j \right\rangle, \quad j
=1,\dots,m .
\end{equation}
Set
$$ \hat{x}_i = \begin{cases}
\hat{x}_{i,m} &; \ 1 \leq i \leq m \\
0 &; \ i > m. 
\end{cases} $$
\item Then we define 
\begin{equation}\label{eq:estgnot}
 \hat{g}_{0,m} ^{(h)}= \sum_{i=1}^\infty \hat{x}_i \psi_i \ \ \left( = \sum\limits_{i=1}^m \hat{x}_i \psi_i \right).
\end{equation}
\end{enumerate}
Equation (\ref{eq:helpmemore}) comes from the fact that for any $f \in V_m$,
$\sum_{i=1}^m \lambda_i \eta_i =\sum_{i=1}^m \left\langle f, e_i \right\rangle e_i$ if and only if 
$\left\langle f, e_i \right\rangle = \sum_{j=1}^m \lambda_j \left\langle \eta_j, e_i \right\rangle$.
Note that $\left\langle e_i, \eta_j \right\rangle = 0$ whenever $i > j$  since $\eta_j$ is a linear combination of $e_1,\dots,e_j$. In particular, formula (\ref{eq:helpmemore}) stays true if $j > m$. Due to that, the system 
of linear equations there becomes diagonal and can easily be solved by backward substitution. \\

\begin{theoremub}
	Let $\bar{g}_1 \in U_A$ and $\hat{\bar{g}}_{1} \in U_A$ be an estimator of $\bar{g}_1$. Let furthermore
	$\hat{\bar{g}}_{1,m} := P_m \hat{\bar{g}}_1$ be the orthogonal projection of $\hat{\bar{g}}_{1}$
	onto $V_m$. Then under the conditions of Theorem \ref{theoSolSec} it holds for $\hat{g}_{0,m}^{(h)}$ as in (\ref{eq:estgnot}) that
	\begin{equation}\label{eq:ErrEst_OnB}
	 \| g_0^{(h)} - \hat{g}_{0,m}^{(h)} \|_\cdot  \leq \frac{|f_1|}{n_1 \big( 1 - e(f,h) \big)} \left[ 2 \left\| \sum\limits_{j=m+1}^\infty x_j \eta_j \right\|_2 + 
	\| \bar{g}_1 - \hat{\bar{g}}_{1,m} \|_\cdot \right],  
	\end{equation}
	where  $x_j=\left\langle {g}_{0}^{(h)}, \psi_j \right\rangle$, $j\in\N$.
\end{theoremub}
\begin{proof}
	First of all, it holds
	\begin{multline*}
		\left( \sum\limits_{j=1}^\infty \left[ \sum\limits_{i=1}^m   x_i \left\langle \eta_i, e_j \right\rangle 
		- \hat{y}_j \right]^2 \right)^{1/2} 
		 = \left( \sum\limits_{j=1}^\infty \left[ \sum\limits_{i=1}^m x_i \left\langle \eta_i, e_j \right\rangle 
		- \sum\limits_{i=1}^m \hat{x}_i \left\langle \eta_i, e_j \right\rangle \right]^2 \right)^{1/2} \\
		 = \left\| \sum\limits_{i=1}^m (x_i - \hat{x}_i) \eta_i \right\|_2 
		 \geq \frac{n_1}{|f_1|}\big( 1 - e(f,h) \big) \left( \sum\limits_{i=1}^m (x_i - \hat{x}_i)^2 \right)^{1/2}, 
	\end{multline*} 
	and therefore
	\begin{align}\label{form:estim} 
	\sum\limits_{i=1}^m (x_i - \hat{x}_i)^2 & \leq \left( \frac{|f_1|}{n_1 \big( 1 - e(f,h) \big)} \right)^2 
	\sum\limits_{j=1}^\infty \left[ \sum\limits_{i=1}^m x_i \left\langle \eta_i, e_j \right\rangle 
	- \hat{y}_j \right]^2 \nonumber \\
	& = \left( \frac{|f_1|}{n_1 \big( 1 - e(f,h) \big)} \right)^2 
	\sum\limits_{j=1}^\infty \left[y_j - \hat{y}_j - \sum\limits_{i=m+1}^\infty x_i \left\langle \eta_i, e_j \right\rangle \right]^2,
	\end{align}  
	with $(y_j)_{j \in \mathbb{N}}$ defined by $y_j = \left\langle \bar{g}_1, e_j \right\rangle = \sum_{i=1}^\infty x_i \left\langle \eta_i, e_j \right\rangle$,
	$j \in \mathbb{N}$, compare \eqref{eq:fp}. Then
	\begin{equation}\label{eq:second}
		 \| g_0^{(h)} - \hat{g}_{0,m}^{(h)} \|_\cdot = 
		\left\| \sum\limits_{j=1}^\infty x_j \psi_j - \sum\limits_{j=1}^m \hat{x}_j \psi_j  \right\|_\cdot \leq \left\| \sum\limits_{j=m+1}^\infty x_j \psi_j  \right\|_2 +
		\left\| \sum\limits_{j=1}^m (x_j - \hat{x}_j) \psi_j  \right\|_\cdot  .
			\end{equation} 
	By (\ref{form:estim}) together with the triangle inequality we get
	\begin{align*}
	\left\| \sum\limits_{j=1}^m (x_j - \hat{x}_j) \psi_j  \right\|_\cdot 
	& = \left( \mathbb{E} \sum\limits_{i=1}^m (x_i - \hat{x}_i)^2 \right)^{1/2} \\
	& \leq \frac{|f_1|}{n_1 \big( 1 - e(f,h) \big)} 
	\left[ \left( \mathbb{E} \sum\limits_{j=1}^\infty (y_j - \hat{y}_j)^2 \right)^{1/2}
	+ \left\| \sum\limits_{i=m+1}^\infty x_i \eta_i\right\|_2 \right] \\
	& =\frac{|f_1|}{n_1 \big( 1 - e(f,h) \big)} 
	\left[   
	\|   \bar{g}_1 - \hat{\bar{g}}_{1,m}\|_\cdot 
	+ \left\| \sum\limits_{i=m+1}^\infty x_i \eta_i\right\|_2 \right] .
	\end{align*}
	Taking into account that
	$ \left\| \sum\limits_{i=m+1}^\infty x_i \eta_i \right\|_2 \geq \frac{n_1}{|f_1|} \big( 1 - e(f,h) \big)
	\left\| \sum\limits_{j=m+1}^\infty x_j \psi_j \right\|_2 $
	the statement of the theorem follows by	(\ref{eq:second}).
\end{proof}

\begin{rem4}
The term $ \left\| \sum\limits_{i=m+1}^\infty x_i \eta_i \right\|_2$ in \eqref{eq:ErrEst_OnB} is the approximation error of $\bar{g}_1=\sum_{i=1}^\infty x_i \eta_i $ by the first $m$ summands of its series.
As $m\to\infty$, the upper bound  \eqref{eq:ErrEst_OnB} tends to $\frac{|f_1|}{n_1 \big( 1 - e(f,h) \big)} 
\|   \bar{g}_1 - \hat{\bar{g}}_{1}\|_\cdot$. In order to estimate $ \bar{g}_1$, the method of Section \ref{sec:est_g_1} can be used if the random field $X$ satisfies the assumptions given there. In this case, Corollaries \ref{cor:ErrEst1} and \ref{cor:ErrEst2} yield an upper bound for $\|   \bar{g}_1 - \hat{\bar{g}}_{1}\|_\cdot$ leading to $L^2$--consistent estimates of $g_0^{(h)}$.
\end{rem4}

Since the estimator in (\ref{eq:estgnot}) is strongly oscillating, a smoothed version $\tilde{g}_{0,m}^{(h)} = \hat{g}_{0,m}^{(h)} \ast K_b$
	of $\hat{g}_{0,m}^{(h)}$ is considered here, where $K_b$ is a smoothing kernel with properties \textbf{(K1)-(K3)} from Section \ref{subsect:PlugIn}. 
	It is clear that $g_0^{(h)} \in L^1(\R)$, $\hat{g}_{0,m}^{(h)} \in L^1(\R) \cap L^2(\R)$, because both are in $U_A$ by assumption. If additionally 
	$g_0^{(h)} \in H^{\delta}(\R)$ for some $\delta > 1/2$ then it immediately follows from the proof of Theorem 
	\ref{thm:ErrorSmoothedPI} that
	\begin{equation*}
	\Vert \tilde{g}_{0,m}^{(h)} - g_0^{(h)} \Vert_\cdot \leq \frac{C_K}{2 \pi} \Vert \hat{g}_{0,m}^{(h)} - g_0^{(h)} \Vert_\cdot +
	||g_0^{(h)}||_{1}^{1/2}  ||g_0^{(h)}||_{H^\delta}^{1/2} a_\delta(b)
	\end{equation*}
	with $a_\delta$ given in (\ref{eq:a_delta}). The bandwidth $b>0$ can be chosen as in Remark \ref{rem:choice_bandwidth}.


\section{Numerical Performance of the Estimators} \label{sect:Num}

\begin{figure}[ht!]
	\subfigure[Plug-In Method]{ \label{fig:1}
		\includegraphics[width = 0.42\textwidth]{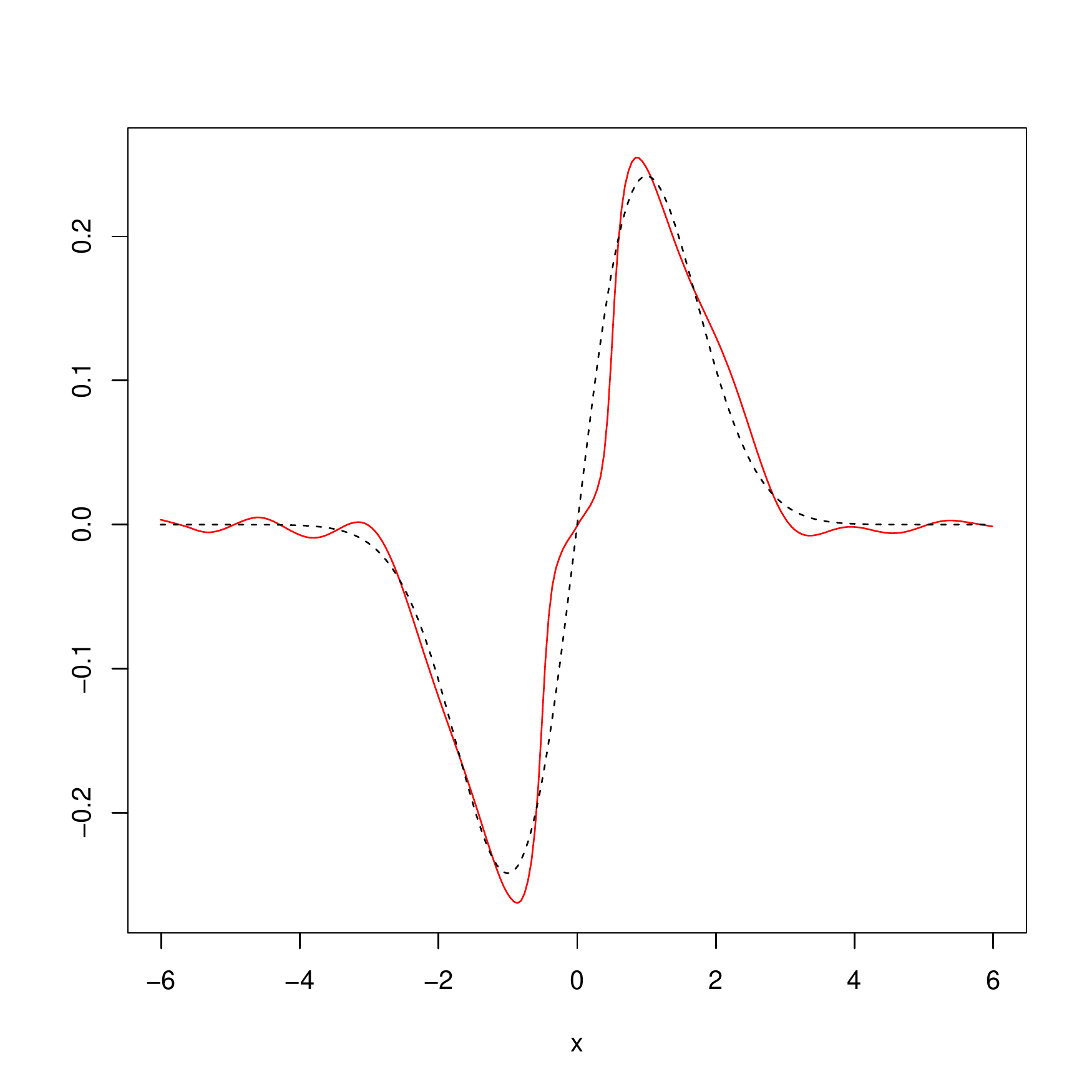}
	} \hfill \subfigure[Fourier Method]{
		\includegraphics[width = 0.42\textwidth]{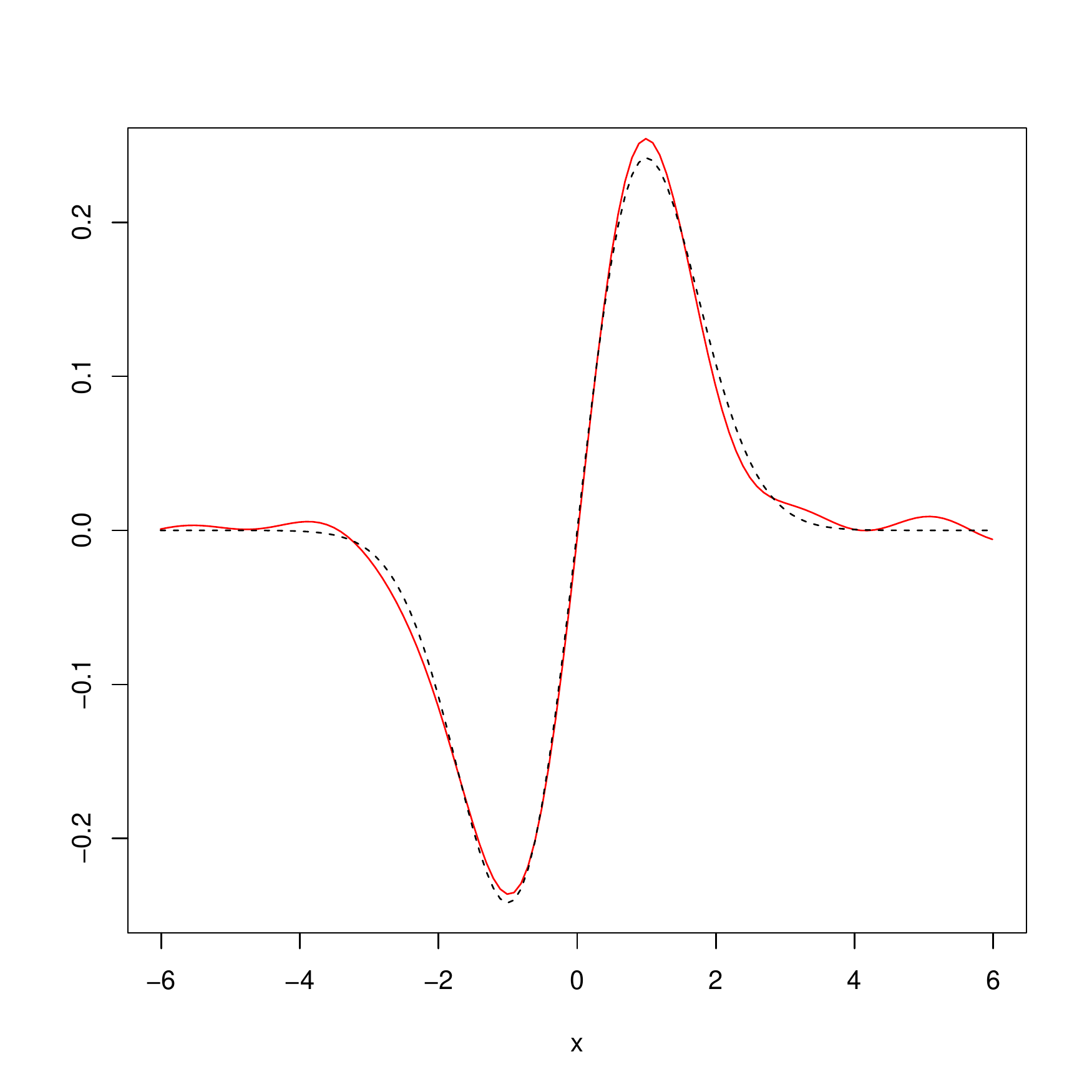}
		\label{fig:2} } 
	\subfigure[OnB Method]{ \label{fig:3}
		\includegraphics[width = 0.42\textwidth]{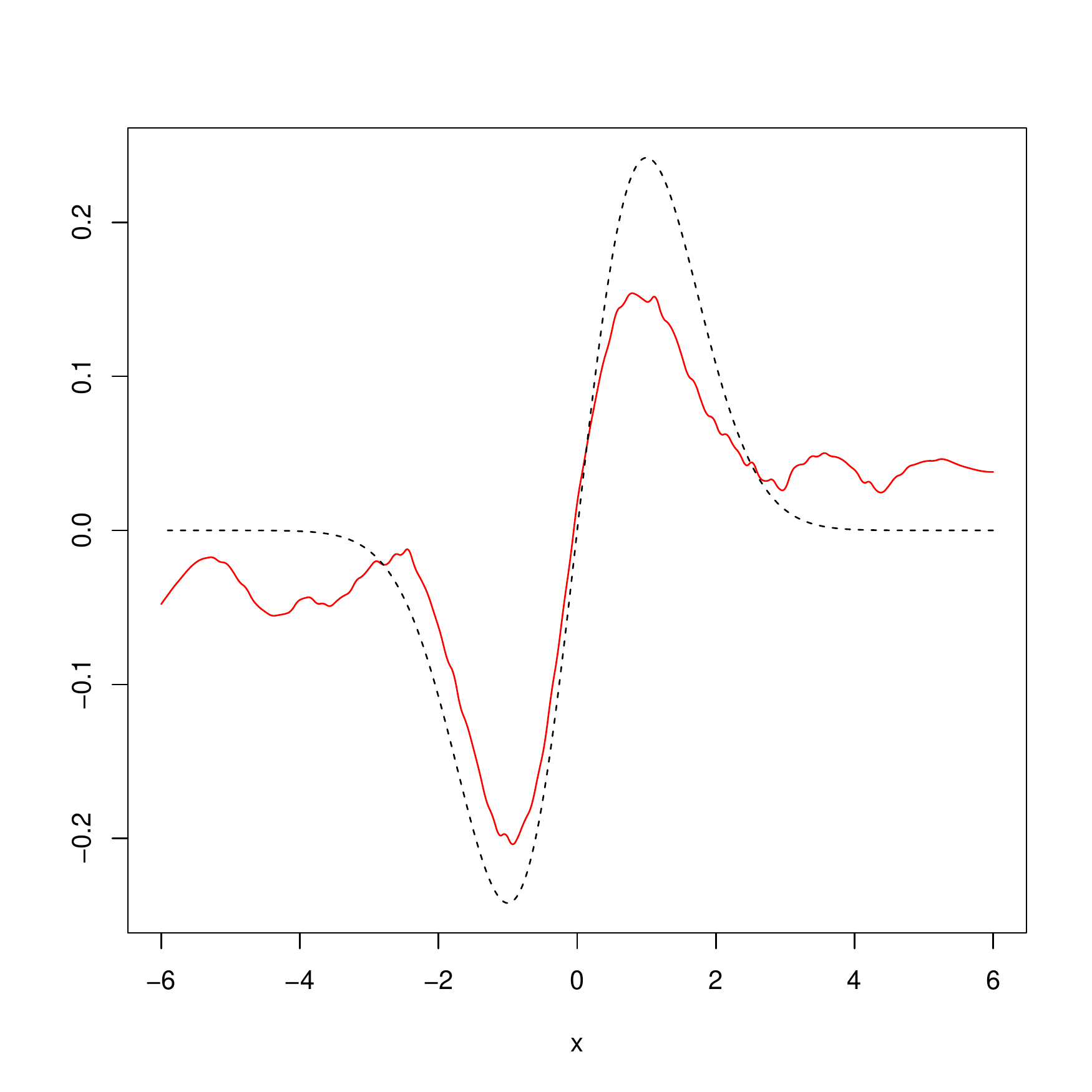}
	}
	\caption{$g_0(x) = (\sqrt{2\pi})^{-1} x \exp(-x^2 / 2 )$ (dashed line)  and a realization of the corresponding estimator (red line) for each of the three
		estimation approaches with sample size $N=10000$.}
	\label{pic:gaussian_v0}
\end{figure}

\begin{figure}[ht!]
	\subfigure[Plug-In Method]{ \label{fig:4}
		\includegraphics[width = 0.42\textwidth]{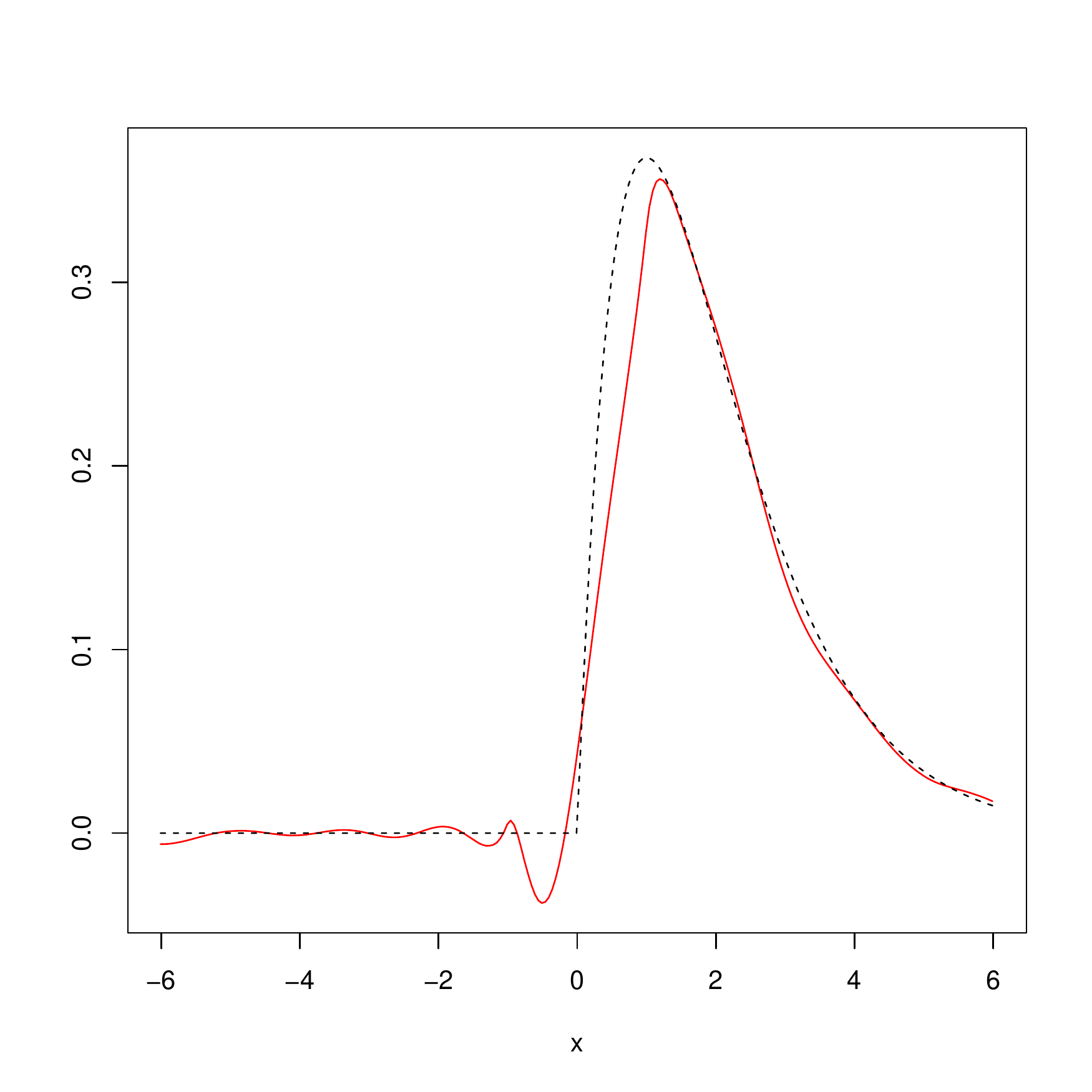}
	} \hfill \subfigure[Fourier Method]{
	\includegraphics[width = 0.42\textwidth]{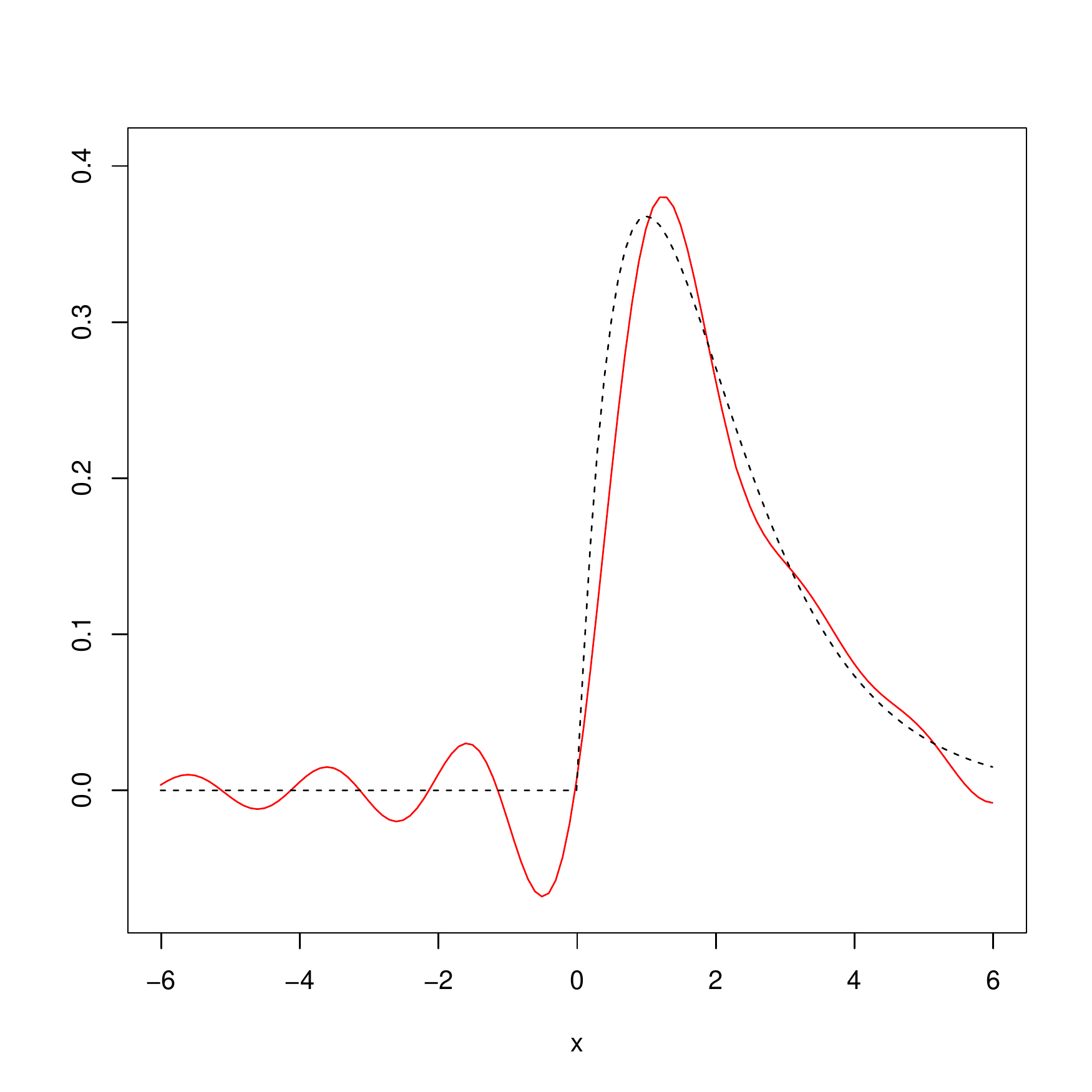}
	\label{fig:5} } 
\subfigure[OnB Method]{ \label{fig:6}
	\includegraphics[width = 0.42\textwidth]{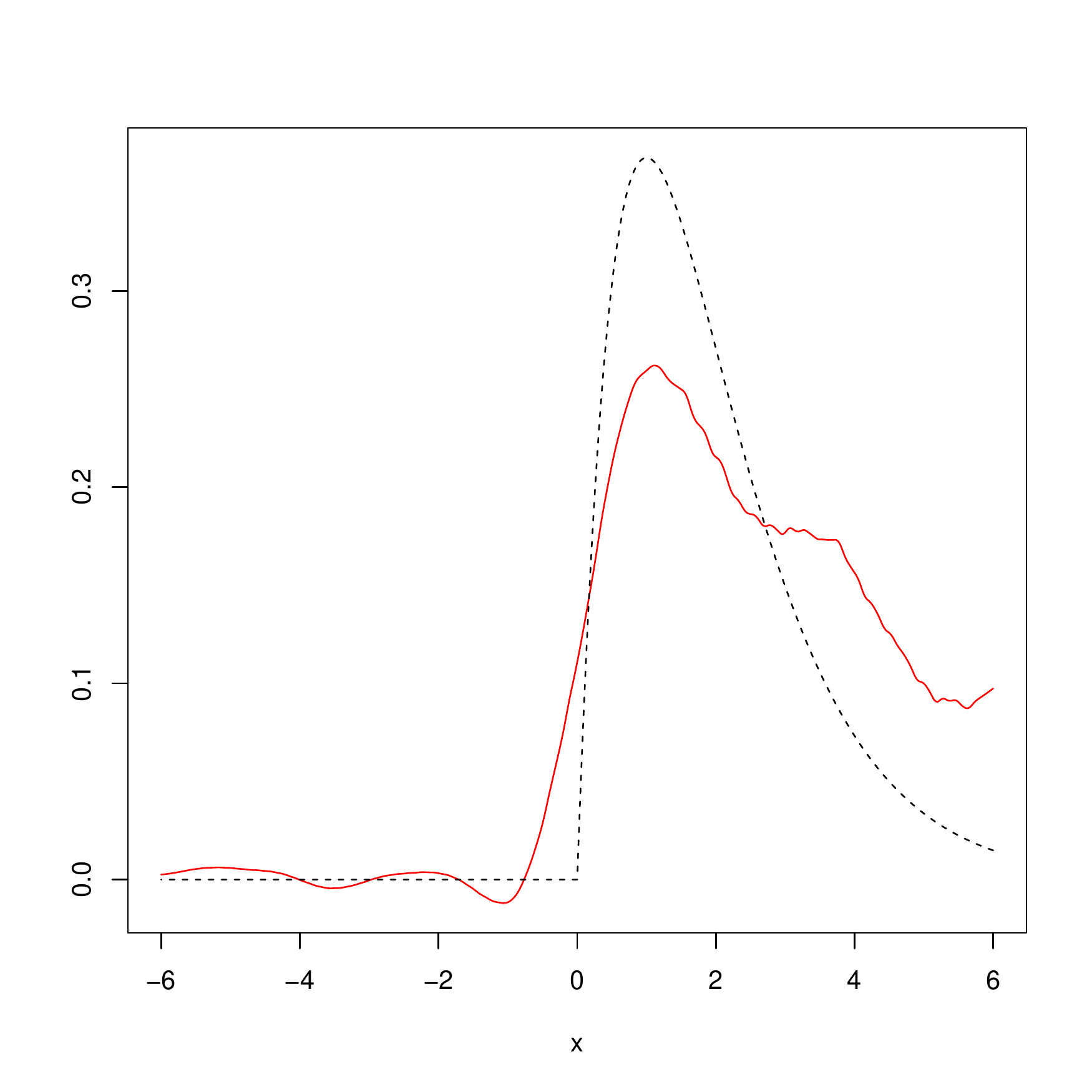}
}
\caption{$g_0(x) = x \exp(-x)\id_{(0,\infty)}(x)$ (dashed line) and a realization of the corresponding estimator (red line) for each of the three
	estimation approaches with sample size $N=10000$.}
\label{pic:exponential_v0}
\end{figure} 

In order to compare the three approaches of Section \ref{sec:es}, we consider $\Lambda(\Delta)$ to be a compound Poisson
random variable
\begin{equation*}
	\Lambda(\Delta) \overset{d}{=} \sum\limits_{k=1}^N Y_k,
\end{equation*}
where $\{Y_k\}_{k \in \mathbb{N}}$ is a sequence of independent and identically distributed random variables,
independent of $N \sim Poi(\nu_d(\Delta))$. Then for any simple function $f = \sum_{k=1}^n f_k \id_{\Delta_k}$
with $\nu_d(\Delta_k) = \nu_d(\Delta)$ for all $k = 1,\dots,n$ it holds
\begin{equation*}
	X(0) \overset{d}{=} \sum\limits_{k=1}^n f_k W_k,
\end{equation*}
where $W_1,\dots,W_n$ are i.i.d. with $W_1 \overset{d}{=} \Lambda(\Delta)$. 

In the following examples, we assumed $d=2$, $n = 4$, $f_1 = 1.3$, $f_2 = 0.2$, $f_3 = f_4 = 0.1$ as well as
$\nu_2(\Delta) = 1$. Then  $v_0$ is the density of the random variable $Y_1$, and due to
 formula \eqref{eq:v_1},  $v_1$ is given by 
\begin{equation*}
	v_1(x) =  \frac{1}{1.3} v_0\left( \frac{x}{1.3} \right) + \frac{1}{0.2} v_0\left( \frac{x}{0.2} \right) + \frac{2}{0.1} v_0\left( \frac{x}{0.1} \right) 
\end{equation*}
 or, equivalently,  
 \begin{equation*}
g_1(x) = g_0\left( \frac{x}{1.3} \right) +  g_0\left( \frac{x}{0.2} \right) + 2 g_0\left( \frac{x}{0.1} \right), 
\end{equation*}
where $g_1(x) = xv_1(x)$ and $g_0(x) = xv_0(x)$, $h(x)=x$. Note that the coefficients $f_1,\dots,f_4$ fulfill conditions of Theorem \ref{th3}, i.e. for given 
$g_1\in L^2(\R)$ there exists a solution $g_0 \in L^2(\R)$ to
the above equation. In our examples, we simulated the random field $X$ on an integer grid. The estimators for $g_0$ based on the corresponding 
sample with sample size $N=10000$ were compared to the original $g_0$  for the following examples:
\begin{align}
	Y_1 \sim N(0,1),\ \text{i.e.} \ v_0(x) & = (\sqrt{2\pi})^{-1} \exp(-x^2 / 2 ) \quad \text{(Fig. \ref{fig:1}-\ref{fig:3}) and} \label{ex:1}\\
	Y_1 \sim \mathsf{Exp} (1),\ \text{i.e.} \ v_0(x) & = \exp(-x) \id_{(0,\infty)}(x) \quad \text{(Fig. \ref{fig:4}-\ref{fig:6})} \label{ex:2}
\end{align}

For the estimators based on the Fourier method from Section \ref{subsect:Fourier}, the parameter $l=1$ is chosen due to 
Corollary \ref{cor:ErrEst2}, cf. Section \ref{sec:est_g_1}. For both the plug-in (Section \ref{subsect:PlugIn}) and the Fourier method, we used 
furthermore the cut-off parameter $n_N = 1$. For the smoothing procedure, the Epanechnikov kernel 
$$K_b(x) = 0.75 b^{-1} (1-(xb^{-1})^2) \id\{ |x|\le b\} $$ 
with bandwidths $b=0.5$ and $b=1.0$  was used in examples  (\ref{ex:1})  and (\ref{ex:2}) respectively, chosen according to Remark \ref{rem:choice_bandwidth}. 
For the OnB method, Haar wavelets $\{ \psi_j\}$ on $[-A,A]$ for $A=6$ were used together with the cut--off parameter $m=7$. The parameter $l>0$ and 
the bandwidth $b>0$ for the estimator in (\ref{eq:estgnot}) (using Epanechnikov kernel $K_b$) were chosen based on a simulation study with different 
parameters. It turned out that visually the best choice for the example in (\ref{ex:1}) is $l=4.5$, $b=0.7$ whereas for the example in (\ref{ex:2}) the 
parameters $l=4.0$, $b=1.1$ turned out to be optimal. Figures \ref{pic:gaussian_v0} and \ref{pic:exponential_v0} show realizations of the estimated
 $g_0$ (red) by our methods compared to the original $g_0$ (dashed) from examples  (\ref{ex:1}) and (\ref{ex:2}).

 The empirical mean and the standard deviation of the mean square errors of our estimation (assessed upon estimation results for $g_0$ out of $100$ 
 simulations of $X$) are given in Table \ref{table:result}. It is seen there that plug-in and Fourier methods perform equally well
 whereas  the mean  error for the OnB method is significantly higher. 
Regarding their computation times (see Table \ref{table:cpuresult}), the Fourier approach outperforms the others since its algorithm  
 is at least $10$ times faster. To summarize, we recommend the Fourier method for the estimation of $v_0$ unless the plug-in approach can be used under milder assumptions on $v_0$ and $v_1$. This  essentially depends on the estimator for $v_1$ which is chosen as a plug-in.

\begin{table}[h]
	\centering
	\begin{tabular}{@{}rrrrr@{}} \toprule
		& & \multicolumn{3}{c}{Method of estimation} \\ \cmidrule(r){3-5}
		& & plug-in  & Fourier & OnB \\ \midrule
		\multirow{2}{2.3cm}{$Y_1 \sim N(0,1)$}& mean & 0.005291606 & 0.0005609035    & 0.02257974 \\
		& sd & 0.0004369446 & 0.0003471337    & 0.001865197 \\ \midrule
		\multirow{2}{2.3cm}{$Y_1 \sim \mathsf{Exp}(1)$}	& mean & 0.1240124   & 0.1306668  & 0.1446655 \\
		& sd & 0.004051844 & 0.005115684 & 0.007453711 \\ \bottomrule
	\end{tabular}
	\caption{Empirical mean and standard deviation of the mean square errors of our estimations based on $100$ simulations.}
	\label{table:result}
\end{table}

\begin{table}[h]
	\centering
	\begin{tabular}{@{}rrrrr@{}} \toprule
		& & \multicolumn{3}{c}{Method of estimation} \\ \cmidrule(r){3-5}
		& & plug-in  & Fourier & OnB \\ \midrule
		\multirow{2}{2.3cm}{$Y_1 \sim N(0,1)$}& mean & 1031.18 & 74.95    & 2726.07 \\
		& sd & 54.19 & 3.66   & 1120.59 \\ \midrule
		\multirow{2}{2.3cm}{$Y_1 \sim \mathsf{Exp}(1)$}	& mean & 1262.24 & 121.24  & 3165.08 \\
		& sd & 13.93 & 18.14 & 721.25 \\ \bottomrule
	\end{tabular}
	\caption{Mean and standard deviation of the estimation CPU-times  (in seconds). The computations were performed on a CPU 
		Intel Xeon E5-2630v3, 2.4 GHz with 128 GB RAM.}
	\label{table:cpuresult}
\end{table}

\section*{Appendix} \label{sec:Appendix}

Here we give a proof of Theorem \ref{TheoGenCat} and its corollaries. Before doing so we prove auxiliary  statements. 

\begin{lem1} \label{Le1}
	Let $Y=\{Y_t, t\in\Z^d\}$ be a random field defined in \eqref{eq:rfY} satisfying ${\rm \bf (H2)}_2$   such that $Y$ is either
	\begin{itemize}
		\item[(i)] $m$-dependent or
		\item[(ii)]  $\phi$-mixing and condition (\ref{Rosphimix}) holds.
	\end{itemize}
Furthermore, let $W\subset \Z^d$ be a finite subset, $N=\card(W)$, and let $\hat{\theta}(u) = \frac{1}{N}\sum_{t\in W}Y_te^{iuY_t}$ and $\theta(u) = \E Y_0e^{iuY_0}$. Then 
	\begin{align*}
	\E\bigl|\hat{\theta}(u)-\theta(u)\bigr|^4 \leq \frac{C}{N^2}\E|Y_0|^4,
	\end{align*}
	where $C>0$ is a constant.
\end{lem1}
\begin{proof}
		It holds that
		{\allowdisplaybreaks
			\begin{align*}
			\E \bigl|\hat{\theta}(u)-\theta(u)\bigr|^4 & =  \frac{1}{N^4}\E\Biggl[\sqrt{\Bigl(\sum_{t\in W}\xi^{(1)}_{t}(u)\Bigr)^2 +\Bigl(\sum_{t\in W}\xi^{(2)}_{t}(u)\Bigr)^2 }\Biggr]^4 \\
			& \leq \frac{2}{N^4}\biggl[\E\Bigl|\sum_{t\in W} \xi^{(1)}_{t}(u)\Bigr|^4 + \E\Bigl|\sum_{t\in W} \xi^{(2)}_{t}(u)\Bigr|^4\biggr].
			\end{align*}}
\begin{itemize}
		\item[(i)]
	    By Theorem \ref{Th1} it holds for $p=4$, $\alpha=1$ and $i=1,2$
		\begin{align*}
		&\E\Bigl|\sum_{t\in W} \xi^{(i)}_{t}(u)\Bigr|^4 \leq \Bigl(8 \sum_{t\in W}b_{t,2}(\xi_{t}^{(i)})\Bigr)^2\\
		& = 2^6 \left( \sum_{t\in W}\Vert \xi_{t}^{(i)}(u)\Vert_2^2 + \underbrace{\sum_{k\in V_t^1}\bigl\Vert\xi^{(i)}_{k}(u)
			\E_{\|k-t\|_\infty}\bigl[\xi^{(i)}_{t}(u)\bigr]\bigr\Vert_1}_{:=D}\right)^2.
		\end{align*}
		To determine expression $D$ it is useful to decompose it into two parts. The first part consists of all $k$ for which $\|k\|_\infty>m$ and the second part contains all other $k$. Hence,
		\begin{align*}
		D = \sum_{\substack{k\in V_t^1 \\ \|k-t\|_\infty> m}} \bigl\Vert\xi^{(i)}_{k}(u) \E_{\|k-t\|_\infty}\bigl[\xi^{(i)}_{t}(u)\bigr]\bigr\Vert_1 +\sum_{\substack{k\in V_t^1 \\ \| k-t\|_\infty\leq m}} \bigl\Vert\xi^{(i)}_{k}(u) \E_{\| k-t\|_\infty}\bigl[\xi^{(i)}_{t}(u)\bigr]\bigr\Vert_1.
		\end{align*}
		For the first part it holds due to $m$--dependence of $\{\xi^{(i)}_{t}(u)\}$ that
			$$
			\sum_{\substack{k\in V_t^1 \\ \|k-t\|_\infty> m}}\bigl\Vert \xi^{(i)}_{k}(u) \E\bigl[\xi^{(i)}_{t}(u)|\F_{V_t^{\|k-t\|_\infty}}\bigr]\bigr\Vert_1 = \sum_{\substack{k\in V_t^1 \\ \|k-t\|_\infty> m}}\bigl\Vert \xi^{(i)}_{k}(u) \E\bigl[\xi^{(i)}_{t}(u)\bigr]\bigr\Vert_1=0,
			$$
since $\xi^{(i)}_{t}(u)$ is centered. Furthermore, for the second sum in expression $D$ it follows by H\"older inequality that
			\begin{align}\label{eq:unten}
			& \sum_{\substack{k\in V_t^1 \\ \|k-t\|_\infty\leq m}} \bigl\Vert\xi^{(i)}_{k}(u) \E\bigl[\xi^{(i)}_{t}(u)|\F_{V_t^{\|k-t\|_\infty}}\bigr]\bigr\Vert_1  \nonumber\\
			 &\leq \!\!\!\!\! \
			 \sum_{\substack{k\in V_t^1 \\ \|k-t\|_\infty\leq m}} \!\!\!\!\!  \bigl\Vert\xi^{(i)}_{k}(u) \bigr\Vert_2 \bigl \Vert\E\bigl[\xi^{(i)}_{t}(u)|\F_{V_t^{\|k-t\|_\infty}}\bigr]\bigr\Vert_2 
			 \leq 
			\bigl\Vert\xi^{(i)}_{ t}(u)\bigr\Vert_2 \!\!\!\!\! \sum_{\substack{k\in V_t^1 \\ \|k-t\|_\infty\leq m}} \!\!\!\!\!\bigl\Vert\xi^{(i)}_{k}(u) \bigr\Vert_2. 
			\end{align}
		Let $\tilde{V}_t^1:=\{k\in V_t^1: \|k-t\|_\infty\leq m\}$ and $n_t:=\card(\tilde{V}_t^1) $.
		 This set is shown in Figure \ref{FigtildeVi1} for $d=2$.
		\begin{figure}[h!]
			\centering
			\includegraphics[width = 0.6\textwidth]{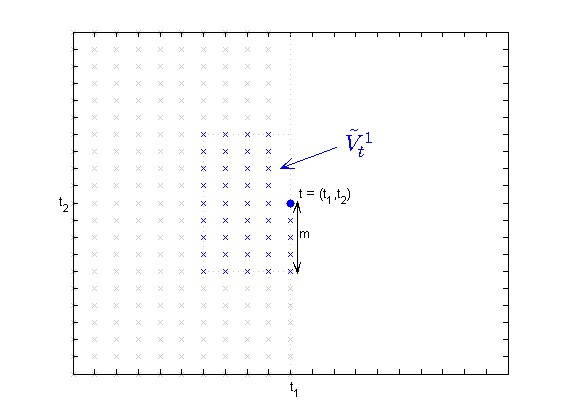}
			\caption[The set $\tilde{V}_t^1=\{k\in V_t^1, \|k-t\|_\infty\leq m\}$.]{The set $\tilde{V}_t^1=\{k\in V_t^1,\, \|k-t\|_\infty\leq m\}$ (blue points) for $m\in\N$.}
			\label{FigtildeVi1}
		\end{figure}
		
		Note that for $i=1,2$ due to stationarity of $Y$
		\begin{align} \label{AbschaetzExihochp}
		\E\bigl|\xi^{(i)}_{t}(u)\bigr|^p \leq 2^p\E|Y_0|^p, \quad p= 1,\ldots, 4.
		\end{align}			
		Therefore, for all $t\in W$ and $i=1,2$ it holds that
		$\bigl\Vert \xi^{(i)}_{t}(u)\bigr\Vert_2\le   2\Vert Y_0 \Vert_2.$
		Applying this to (\ref{eq:unten}), we get 
		$
		\bigl\Vert\xi^{(i)}_{t}(u)\bigr\Vert_2 \sum_{k\in \tilde{V}_t^1} \bigl\Vert\xi^{(i)}_{k}(u) \bigr\Vert_2\leq  4n_t\Vert Y_0\Vert_2^2.
		$
		Moreover, it follows
		\begin{align}  \nonumber
		\E&\Bigl|\sum_{t\in W} \xi^{(i)}_{t}(u)\Bigr|^4 \leq 2^6\Bigl(\sum_{t\in W}\bigl( 2\Vert Y_0 \Vert_2^2 + 4n_t\Vert Y_0\Vert_2^2\bigr)\Bigr)^2 \\ 
		&\leq 2^6\Bigl(2N \Vert Y_0 \Vert_2^2 + 4n^*N\Vert Y_0\Vert_2^2\Bigr)^2=2^8 N^2
		\Vert Y_0 \Vert_2^4(1+2n^*)^2 \label{Gl11}
		\end{align}
		with $n^*:=\underset{t\in W}{\max} \{n_t\}$.
		By Ljapunov inequality, it holds
		\begin{align*}
		\E\bigl|\hat{\theta}(u)-\theta(u)\bigr|^4 \leq  \frac{2^{10}}{N^2}(1+2n^*)^2 \E|Y_0|^4.
		\end{align*}
		\item[(ii)] Using Theorem \ref{Th3} with $p=4$ and applying the Ljapunov inequality we get 
		\begin{align*}
		&\E\Bigl|\sum_{t\in W} \xi^{(i)}_{t}(u)\Bigr|^4 \leq C_i\cdot \max\biggl\{\sum_{t\in W}\E\bigl|\xi^{(i)}_{t}(u)\bigr|^4,\Bigl(\sum_{t\in W}\E\bigl|\xi^{(i)}_{t}(u)\bigr|^2\Bigr)^2\biggr\}\\
		&\leq C_i\cdot \max\biggl\{\sum_{t\in W}\E\bigl|\xi^{(i)}_{t}(u)\bigr|^4,\Bigl(\sum_{t\in W}\bigl(\E\bigl|\xi^{(i)}_{t}(u)\bigr|^4\bigr)^{1/2}\Bigr)^2\biggr\}\\
		&=C_i\cdot \max\Bigl\{N\E\bigl|\xi^{(i)}_{0}(u)\bigr|^4,N^2\E\bigl|\xi^{(i)}_{0}(u)\bigr|^4\Bigr\} = C_i N^2\E\bigl|\xi^{(i)}_{0}(u)\bigr|^4 \leq 2^4C_i N^2\E|Y_0|^4
		\end{align*}
		for some constants $C_i>0$, $i=1,2$, where the last inequality follows by equation (\ref{AbschaetzExihochp}).
		Thus, we have
		\begin{align*}
		\E|\hat{\theta}(u)-\theta(u)|^4 \leq \frac{2}{N^4}\Bigl[2^4C_1N^2\E|Y_0|^4+ 2^4C_2N^2\E|Y_0|^4\Bigr] = \frac{C}{N^2}\E|Y_0|^4,
		\end{align*}
		where $C = 2^5(C_1+C_2)>0$ is constant.
	\end{itemize}
\end{proof}

If assumption (i) holds then the constant $C$ is given by $C = 2^{10}(1+2n^*)^2$, where $n^*\le m^d$ is the maximum over the cardinalities of the sets $\tilde{V}_t^1$ for every $t\in W$. Therefore, in the first case the constant $C$ depends on $m$. In the second case the constant $C = 2^5(C_1+C_2)$ depends on the   mixing coefficient $\phi_{u,v}(r)$ by Theorem \ref{Th3}.

\begin{lem2}
	\label{Le2}
	Let $\hat{\psi}(u) = \frac{1}{N}\sum_{t\in W}e^{iuY_t}$ and $\psi(u) = \E e^{iuY_0}$ where $N=\card(W)$. Under the assumptions of Lemma \ref{Le1}  for $p\geq2$ there exists a constant $C_p>0$ such that
	\begin{align*}
	\E\bigl|\hat{\psi}(u) - \psi(u)\bigr|^p\leq \frac{C_p}{N^{p/2}}.
	\end{align*}
\end{lem2}
\begin{proof}
	Since $x\mapsto |x|^p$, $p\ge 2$ is a convex function  it holds
		\begin{align}\nonumber
		\E\bigl|\hat{\psi}(u)-\psi(u)\bigr|^p 
		& = \frac{1}{N^p}\E\Bigl|\Bigl(\sum_{t\in W}\tilde{\xi}^{(1)}_{t}(u)\Bigr)^2 + \Bigl(\sum_{t\in W}\tilde{\xi}^{(2)}_{t}(u)\Bigr)^2\Bigr|^{p/2}\\ \label{Abschaetzungpsi}
		& \leq \frac{2^{p/2-1}}{N^p}\biggl[\E\Bigl|\sum_{t\in W}\tilde{\xi}^{(1)}_{t}(u)\Bigr|^p +\E\Bigl|\sum_{t\in W}\tilde{\xi}^{(2)}_{t}(u)\Bigr|^p \biggr].	
		\end{align}
		
		\begin{itemize}
		\item[(i)]
		Applying Theorem \ref{Th1} with  $\alpha=1$ we get for $i=1,2$
		\begin{align*}
		\E\Bigl|\sum_{t\in W}\tilde{\xi}^{(i)}_{t}(u)\Bigr|^p  \leq \Bigl(2p\sum_{t\in W}\bigl(\bigl\Vert\tilde{\xi}^{(i)}_{t}(u)\bigr\Vert_{2}^2 
		+ \sum_{k\in V_t^1}\bigl\Vert \tilde{\xi}^{(i)}_{k}(u)\E_{\|k-t\|_\infty}\bigl[\tilde{\xi}^{(i)}_{t}(u)\bigr]\bigr\Vert_{1}\bigr)\Bigr)^{p/2}.
		\end{align*}
		Since 
		\begin{align}
		\label{AbschaetzExihochp2}
		\bigl|\tilde{\xi}^{(i)}_{t}(u)\bigr|\leq 2   \quad \mbox{ a.s. for all } t\in \Z^d, \; u\in \R,\; i=1,2
		\end{align}
		it follows
		$
		\bigl\Vert\tilde{\xi}^{(i)}_{ t}(u)\bigr\Vert_{2}^2 \leq  4.
		$
		Analogously to the calculations in the proof of Lemma \ref{Le1} (i) we observe
			\begin{align*}
			&\sum_{k\in V_t^1}\bigl\Vert \tilde{\xi}^{(i)}_{k}(u)\E_{\|k-t\|_\infty}\bigl[\tilde{\xi}^{(i)}_{t}(u)\bigr]\bigr\Vert_{1}  
			= \sum_{k\in \tilde{V}_t^1}\bigl\Vert \tilde{\xi}^{(i)}_{k}(u)\E\bigl[\tilde{\xi}^{(i)}_{t}(u)|\F_{V_t^{\|k-t\|_\infty}}\bigr]\bigr\Vert_{1}\leq 4n_t,
			\end{align*}
		and hence
		\begin{equation*}
		\E\Bigl|\sum_{t\in W}\tilde{\xi}^{(i)}_{t}(u)\Bigr|^p \leq \Bigl(2p\sum_{t\in W} (4+4n_t)\Bigr)^{p/2} \leq 2^{\frac{3}{2}p}p^{p/2}(N(1+n^*))^{p/2}.
		\end{equation*}
		So all in all we get from \eqref{Abschaetzungpsi} that
		$
		\E \bigl|\hat{\psi}(u) - \psi(u)\bigr|^p \leq  \frac{C_p}{N^{p/2}}
		$
		for the constant $C_p=2^{2p} p^{p/2} (1+n^*)^{p/2}>0$ and $n^*\le m^d$.
		
		\item[(ii)] Using Theorem \ref{Th3} and  inequality \eqref{AbschaetzExihochp2} it follows for $p\geq 2$ and $i=1,2$ that
		{\allowdisplaybreaks
			\begin{align*}
			&\E\Bigl|\sum_{t\in W}\tilde{\xi}^{(i)}_{t}(u)\Bigr|^p \leq C_i\cdot \max\biggl\{\sum_{t\in W}\E\bigl|\tilde{\xi}^{(i)}_{t}(u)\bigr|^p,\Bigl(\sum_{t\in W}\E\bigl|\tilde{\xi}^{(i)}_{t}(u)\bigr|^2\Bigr)^{p/2}\biggr\} \\
			& \leq C_i\cdot \max\Bigl\{N\E\bigl|\tilde{\xi}^{(i)}_{t}(u)\bigr|^p,N^{p/2}\left( \E\bigl|\tilde{\xi}^{(i)}_{t}(u)\bigr|^2 \right)^{p/2} \Bigr\} \le  C_i N^{p/2}2^p.
			\end{align*}}
		By equation (\ref{Abschaetzungpsi}) it finally follows
		$
		\E \bigl|\hat{\psi}(u)-\psi(u)\bigr|^p \leq \frac{C_p}{N^{p/2}},
		$
		where $C_p = 2^{\frac{3}{2}p-1}(C_1+C_2)>0$ is a constant depending on $p$ and the  mixing coefficient of $\{  \tilde{\xi}_t^{(i)} \}$, $i=1,2$.
	\end{itemize}
\end{proof} 
The following lemma is a generalization of \cite[Lemma 2.1]{hoessjer} (proven there for independent random variables and $p=1$) to the case of weakly dependent random fields.
\begin{lem3} \label{Le3}
	Under the assumptions of Lemma \ref{Le1} together with condition (\ref{eq:B})  there exists a constant $C>0$ such that for $p\in \N$ 
	\begin{align}
	\E\left|\frac{1}{\tilde{\psi}(u)}-\frac{1}{\psi(u)}\right|^{2p} \leq C\cdot \min\Biggl\{\frac{N^{-p}}{|\psi(u)|^{4p}},\frac{1}{|\psi(u)|^{2p}}\Biggr\}.
	\end{align}
\end{lem3}
\begin{proof}
	\begin{itemize}
		\item[1.)] Let $|\psi(u)|<2N^{-1/2}$. Then it holds
			\begin{align*}
			\E &\left|\frac{1}{\tilde{\psi}(u)} -\frac{1}{\psi(u)}\right|^{2p} = \E \left|\frac{\id\{|\hat{\psi}(u)|\geq N^{-1/2}\}}{\hat{\psi}(u)}-\frac{1}{\psi(u)}\right|^{2p} \\
			&= \E\left|\frac{\id\{|\hat{\psi}(u)|\geq N^{-1/2}\}\cdot (\psi(u)-\hat{\psi}(u))}{\hat{\psi}(u)\psi(u)}+\frac{ \hat{\psi}(u)\id\{|\hat{\psi}(u)|< N^{-1/2}\}}{\hat{\psi}(u)\psi(u)} \right|^{2p}\\
			& \leq 2^{2p-1}\Biggl(\frac{1}{|\psi(u)|^{2p}} P(|\hat{\psi}(u)|< N^{-1/2}) + \E \left[\frac{|\psi(u)- \hat{\psi}(u)|^{2p}} {|\hat{\psi}(u)|^{2p}|\psi(u)|^{2p}}\id\{|\hat{\psi}(u)|\geq N^{-1/2}\}\right]\Biggr)\\
			& \leq 2^{2p-1}\Biggl(\frac{1}{|\psi(u)|^{2p}} + \frac{C_{2p} N^{-p}}{N^{-p}|\psi(u)|^{2p}} \Biggr)= \mathcal{O}\left(\frac{1}{|\psi(u)|^{2p}}\right),
			\end{align*}
		where the last inequality follows by Lemma \ref{Le2} and the fact that an indicator is always smaller or equal than $1$.
		In this case, we get for $|\psi(u)|<2N^{-1/2}$ that
		\begin{align*}
		\frac{N^{-p}}{|\psi(u)|^{4p}} = \frac{N^{-p}}{|\psi(u)|^{2p}}\cdot\frac{1}{|\psi(u)|^{2p}}> \frac{N^{-p}}{|\psi(u)|^{2p}}\cdot\frac{N^p}{2^{2p}} = \frac{2^{-2p}}{|\psi(u)|^{2p}}.
		\end{align*}
		\item[2.)] Let $|\psi(u)|\geq 2N^{-1/2}$. Then we get
		{\allowdisplaybreaks
			\begin{align}\nonumber
			P& (|\hat{\psi}(u)|<N^{-1/2}) 
			= P(|\psi(u)|-|\hat{\psi}(u)|>|\psi(u)|-N^{-1/2})\\ \nonumber
			&\leq P\Bigl(|\hat{\psi}(u)-\psi(u)|>\frac{1}{2}|\psi(u)|\Bigr) = P\Bigl(\Bigl|\sum_{t\in W}\bigl(e^{iuY_t}-\E e^{iuY_0}\bigr)\Bigr|> \frac{N}{2}|\psi(u)|\Bigr)\\ \nonumber
			& = P \Bigl( \sqrt{\Bigl(\sum_{t\in W}\tilde{\xi}_{t}^{(1)}(u)\Bigr)^2 + \Bigl(\sum_{t\in W}\tilde{\xi}_{t}^{(2)}(u)\Bigr)^2}> \frac{N}{2}|\psi(u)|\Bigr) \\ \nonumber
			&\leq P \Bigl(\max_{i=1,2}\Bigl\{\Bigl|\sum_{t\in W} \tilde{\xi}_{t}^{(i)}(u)\Bigr|\Bigr\}>\frac{N}{2\sqrt{2}}|\psi(u)|\Bigr)\\ 
			& \leq P  \Bigl(\Bigl|\sum_{t\in W} \tilde{\xi}_{t}^{(1)}(u)\Bigr|>\frac{N}{2\sqrt{2}}|\psi(u)|\Bigr) + P \Bigl(\Bigl|\sum_{t\in W} \tilde{\xi}_{ t}^{(2)}(u)\Bigr|>\frac{N}{2\sqrt{2}}|\psi(u)|\Bigr).  \label{GlnachBernstein}
			\end{align}}
		To calculate this probability, we consider assumptions (i) and (ii) separately.
		\begin{itemize}
			\item[(i)] Here we can apply Theorem \ref{Th2} and we get for $i=1,2$
			\begin{align*}
			P  \Bigl(\Bigl|\sum_{t\in W} \tilde{\xi}_{t}^{(i)}(u)\Bigr|>\frac{N}{2\sqrt{2}}|\psi(u)|\Bigr)\leq \exp\biggl\{\frac{1}{e}- \frac{\bigl(\frac{N}{2\sqrt{2}}|\psi(u)|\bigr)^2}{4eb_i}\biggr\},
			\end{align*}
			where 
			\begin{align*}
			b_i = \sum_{t\in W} b_{t,\infty}(\tilde{\xi}^{(i)})
			= \sum_{t\in W} \bigl(\bigl\Vert \bigl(\tilde{\xi}_{t}^{(i)}(u)\bigr)^2\bigr\Vert_\infty + \sum_{k\in V_t^1} \bigl\Vert \tilde{\xi}_{k}^{(i)}(u) \E_{\|k-t\|_\infty} \bigl[\tilde{\xi}_{t}^{(i)}(u)\bigr]\bigr\Vert_\infty\bigr)
			\end{align*}
			and $\Vert Z \Vert_\infty := \inf\{c>0 : P(|Z|>c)=0\}$ for a random variable $Z$.
			By inequality \eqref{AbschaetzExihochp2}  and $m$-dependence
				\begin{align*}
				&\sum_{k\in V_t^1}\bigl\Vert \tilde{\xi}_{k}^{(i)}(u)\E_{\|k-t\|_\infty} \bigl[\tilde{\xi}_{t}^{(i)}(u)\bigr]\bigr\Vert_\infty 
				=  \sum_{k\in \tilde{V}_t^1}\bigl\Vert \tilde{\xi}_{k}^{(i)}(u) \E_{\|k-t\|_\infty} \bigl[\tilde{\xi}_{t}^{(i)}(u) \bigr]\bigr\Vert_\infty \le 4n_t.
				\end{align*}	
			Therefore, $b_i$ can be estimated as 
			$
			b_i \leq \sum_{t\in W}(4+4n_t)\leq 4N(1+n^*),
			$ $i=1,2$,
			with $n^*$ as in the proof of Lemma \ref{Le1}.
			For expression (\ref{GlnachBernstein}) we get
			\begin{align*}
			&P  \Bigl(\Bigl|\sum_{t\in W} \tilde{\xi}_{t}^{(1)}(u)\Bigr|>\frac{N}{2\sqrt{2}}|\psi(u)|\Bigr) + P \Bigl(\Bigl|\sum_{t\in W} \tilde{\xi}_{ t}^{(2)}(u)\Bigr|>\frac{N}{2\sqrt{2}}|\psi(u)|\Bigr)\\
			&\leq  2\cdot \exp\biggl\{\frac{1}{e}-\frac{N|\psi(u)|^2}{128e(1+n^*)}\biggr\}= \mathcal{O}\Bigl(\frac{N^{-p}}{|\psi(u)|^{2p}}\Bigr).
			\end{align*}
			
			\item[(ii)] Apply Theorem \ref{Th4} to  $\{  \tilde{\xi}_t^{(i)}(u) \}$ with $a_t=1$ for all $t\in W$ and $h=2$. Then, $A(W)=N$, and we  have						\begin{align*}
				&P  \Bigl(\Bigl|\sum_{t\in W} \tilde{\xi}_{t}^{(1)}(u)\Bigr|>\frac{N}{2\sqrt{2}}|\psi(u)|\Bigr) + P \Bigl(\Bigl|\sum_{t\in W} \tilde{\xi}_{ t}^{(2)}(u)\Bigr|>\frac{N}{2\sqrt{2}}|\psi(u)|\Bigr)\\ 
				&\leq 2\cdot \exp\Bigl\{\frac{1}{e}-\frac{\frac{N^2}{8}|\psi(u)|^2}{16(1+B(\phi))Ne}\Bigr\} = 2\cdot \exp \Bigl\{\frac{1}{e}- \frac{N|\psi(u)|^2}{128(1+B(\phi))e}\Bigr\}\\& = \mathcal{O}\Bigl(\frac{N^{-p}}{|\psi(u)|^{2p}}\Bigr).
				\end{align*}
			
		\end{itemize}
		So we get in both cases
		\begin{align*}
		P (|\hat{\psi}(u)|<N^{-1/2}) = \mathcal{O}\Bigl(\frac{N^{-p}}{|\psi(u)|^{2p}}\Bigr).
		\end{align*}
		
		It holds that 
			\begin{align*}
			&\frac{1}{|\hat{\psi}(u)|^{2p}}= \frac{|\psi(u)|^{2p}}{|\psi(u)|^{2p}|\hat{\psi}(u)|^{2p}}= \left(\frac{|\psi(u)-\hat{\psi}(u)+\hat{\psi}(u)|^2}{|\psi(u)|^2|\hat{\psi}(u)|^2}\right)^p\\
			&\leq \left(\frac{1}{|\psi(u)|^2}+\frac{|\hat{\psi}(u)-\psi(u)|^2} {|\hat{\psi}(u)|^2|\psi(u)|^2}\right)^p = \frac{1}{|\psi(u)|^{2p}}\left(1+\frac{|\hat{\psi}(u)-\psi(u)|^2 }{|\hat{\psi}(u)|^2}\right)^p.
			\end{align*}
		Applying the binomial theorem and $|\hat{\psi}(u)|\geq N^{-1/2}$ we get
			\begin{align*}
			&\frac{1}{|\psi(u)|^{2p}}\left(1+\frac{|\hat{\psi}(u)-\psi(u)|^2 }{|\hat{\psi}(u)|^2}\right)^p = \frac{1}{|\psi(u)|^{2p}}\sum_{k=0}^p\binom{p}{k}\frac{|\hat{\psi}(u)-\psi(u)|^{2k}}{|\hat{\psi}(u)|^{2k}}\\
			&\leq \frac{1}{|\psi(u)|^{2p}}\sum_{k=0}^p\binom{p}{k} \frac{|\hat{\psi}(u)-\psi(u)|^{2k}}{N^{-k}}. 
			\end{align*}
		Therefore,
			\begin{align*}
			&\E \Biggl[\frac{|\hat{\psi}(u)-\psi(u)|^{2p}}{|\hat{\psi} (u)|^{2p}|\psi(u)|^{2p}}\id\{|\hat{\psi}(u)|\geq N^{-1/2}\}\Biggr] \leq \frac{1}{|\psi(u)|^{4p}}\biggl[\sum_{k=0}^p\binom{p}{k}\E|\hat{\psi}(u)-\psi(u)|^{2k+2p}N^k\biggr]\\
			&\leq \frac{1}{|\psi(u)|^{4p}}\biggl[\sum_{k=0}^p \binom{p}{k}C_{2k+2p} N^{-k-p}N^k\biggr] = \mathcal{O}\Bigl(\frac{N^{-p}}{|\psi(u)|^{4p}}\Bigr).
			\end{align*}
		So all in all, it holds
			\begin{align*}
			&\E \left|\frac{\id\{|\hat{\psi}(u)|\geq N^{-1/2}\}}{\hat{\psi}(u)}-\frac{1}{\psi(u)}\right|^{2p}\leq \frac{1}{|\psi(u)|^{2p}}P(|\hat{\psi}(u)|<N^{-1/2})\\ 
			&+\E \Biggl[\frac{|\hat{\psi}(u)- \psi(u)|^{2p}}{|\hat{\psi} (u)|^{2p}|\psi(u)|^{2p}} \id\{|\hat{\psi}(u)|\geq N^{-1/2}\}\Biggr]
			= \mathcal{O}\Bigl(\frac{N^{-p}}{|\psi(u)|^{4p}}\Bigr),
			\end{align*}
		that concludes the proof.
	\end{itemize}
\end{proof}

Now we can finalize the proof of Theorem \ref{TheoGenCat}.

\begin{proof}[Proof of Theorem \ref{TheoGenCat}]
	Note that $g_1 -g_{1,l}$ is orthogonal to $\hat{g}_{1,l}-g_{1,l}$, since
	\begin{align*}
	&\langle g_1-g_{1,l},\hat{g}_{1,l}-g_{1,l}\rangle = \langle g_1,\hat{g}_{1,l} \rangle - \langle g_1, g_{1,l} \rangle - \langle g_{1,l},\hat{g}_{1,l} \rangle + \langle g_{1,l}, g_{1,l} \rangle \\
	& =\frac{1}{2\pi}\bigl(\langle \F[g_1],\F[\hat{g}_{1,l}]\rangle - \langle \F[g_1], \F[g_{1,l}]\rangle - \langle \F[g_{1,l}],\F[\hat{g}_{1,l}] \rangle + \langle \F[g_{1,l}],\F[g_{1,l}]\rangle\bigr)\\
	&=0
	\end{align*}
	due to isometry property of $\F$ in $L^2(\R)$.  By Pythagorean theorem we get
	\begin{align*}
	\lVert g_1 - \hat{g}_{1,l}\rVert_2^2 = \lVert g_1 - g_{1,l} \rVert_2^2+\lVert g_{1,l}-\hat{g}_{1,l}\rVert_2^2,			
	\end{align*}	
	and the second term can further be determined by
	\begin{align*}
	\lVert \hat{g}_{1,l}-g_{1,l} \rVert_2^2 = \frac{1}{2\pi}\bigl\lVert\F[\hat{g}_{1,l}]-\F[g_{1,l}]\bigr\rVert_2^2 = \frac{1}{2\pi}\int_{-\pi l}^{\pi l}\biggl|\frac{\hat{\theta}(x)}{\tilde{\psi}
		(x)}-\frac{\theta(x)}{\psi(x)}\biggr|^2dx.
	\end{align*}
	Furthermore,
		\begin{align*}
		&\lVert  \hat{g}_{1,l}-g_{1,l} \rVert_\cdot^2 = \frac{1}{2\pi}\int_{-\pi l}^{\pi l} \E \biggl|\frac{\hat{\theta}(x)}{\tilde{\psi}(x)} -\frac{\hat{\theta}(x)}{\psi(x)}+\frac{\hat{\theta}(x)}{\psi(x)}-\frac{\theta(x)}{\psi(x)}\biggr|^2dx\\
		&\leq \frac{1}{\pi}\left[\int_{-\pi l}^{\pi l}\E\biggl|\hat{\theta}(x)\Big(\frac{1}{\tilde{\psi}(x)}- \frac{1}{\psi(x)}\Big)\biggr|^2dx + \int_{-\pi l}^{\pi l}\frac{\E|\hat{\theta}(x)-\theta(x)|^2}{|\psi(x)|^2}dx\right]\\
		&=\frac{1}{\pi}\left[\int_{-\pi l}^{\pi l}\E\biggl[ |\hat{\theta}(x)-\theta(x)+\theta(x)|^2\biggl|\frac{1}{\tilde{\psi}(x)}- \frac{1}{\psi(x)}\biggr|^2\biggl]dx + \int_{-\pi l}^{\pi l}\frac{\E|\hat{\theta}(x)-\theta(x)|^2}{|\psi(x)|^2}dx\right]\\
		&\leq \frac{1}{\pi}\Biggl[2\int_{-\pi l}^{\pi l}\E\biggl[ |\hat{\theta}(x)-\theta(x)|^2\biggl|\frac{1}{\tilde{\psi}(x)}- \frac{1}{\psi(x)}\biggr|^2\biggr]dx + 2\int_{-\pi l}^{\pi l}\E\biggl[ |\theta(x)|^2\biggl|\frac{1}{\tilde{\psi}(x)}- \frac{1}{\psi(x)}\biggr|^2\biggr]dx\\
		&\qquad + \int_{-\pi l}^{\pi l}\frac{\E|\hat{\theta}(x)-\theta(x)|^2}{|\psi(x)|^2}dx\Biggr]\\
		&= \frac{2}{\pi}\Biggl[\underbrace{\int_{-\pi l}^{\pi l}\E\biggl[ |\hat{\theta}(x)-\theta(x)|^2\biggl|\frac{1}{\tilde{\psi}(x)}- \frac{1}{\psi(x)}\biggr|^2\biggr]dx}_{=:\text{(I)}} \\
		&\qquad + \underbrace{\int_{-\pi l}^{\pi l}|\F[g_1](x)\cdot\psi(x)|^2\E\biggl|\frac{1}{\tilde{\psi}(x)}- \frac{1}{\psi(x)}\biggr|^2dx}_{=:\text{(II)}} + \underbrace{\int_{-\pi l}^{\pi l}\frac{\E|\hat{\theta}(x)-\theta(x)|^2}{2|\psi(x)|^2}dx}_{=:\text{(III)}}\Biggr].\\
		\end{align*}
	First we calculate expression (I). Using the Cauchy-Schwarz inequality and applying Lemma \ref{Le1} and Lemma \ref{Le3} it holds
	\begin{align*}
	&\E\biggl[ |\hat{\theta}(x)-\theta(x)|^2\biggl|\frac{1}{\tilde{\psi}(x)}- \frac{1}{\psi(x)}\biggr|^2\biggr] \leq \sqrt{\E|\hat{\theta}(x)-\theta(x)|^4}\sqrt{\E\biggl|\frac{1}{\tilde{\psi}(x)}- \frac{1}{\psi(x)}\biggr|^4}\\
	&\leq \sqrt{\frac{c_1\cdot c_2}{N^2}\cdot \frac{\E|Y_0|^4}{|\psi(x)|^4}},
	\end{align*}
	where $c_1,c_2>0$ are some constants.
	Now we get
	\begin{align*}
	&\int_{-\pi l}^{\pi l}\E\biggl[ |\hat{\theta}(x)-\theta(x)|^2\biggl|\frac{1}{\tilde{\psi}(x)}- \frac{1}{\psi(x)}\biggr|^2\biggr]dx \leq  \frac{\sqrt{c_1\cdot c_2}}{N}\sqrt{\E|Y_0|^4}\int_{-\pi l}^{\pi l}\frac{1}{|\psi(x)|^2}dx.
	\end{align*}
	For the second term (II) we use again Lemma \ref{Le3}. Then it holds
	\begin{align*}
	\E\left|\frac{1}{\tilde{\psi}(u)}-\frac{1}{\psi(u)}\right|^{2} \leq \frac{1}{N}\cdot\frac{c}{|\psi(x)|^4},
	\end{align*}		
	for some $c>0$. Using this, we get
		\begin{align}\label{eq:F_psi}
		&\int_{-\pi l}^{\pi l}|\F[g_1](x)\cdot\psi(x)|^2\,\,\E\biggl|\frac{1}{\tilde{\psi}(x)}- \frac{1}{\psi(x)}\biggr|^2dx \leq \frac{c}{N}\int_{-\pi l}^{\pi l}|\F[g_1](x)|^2\frac{1}{|\psi(x)|^2}dx \\
		& \leq \frac{c}{N}\int_{-\pi l}^{\pi l}\frac{\left(\int_{\R}|g_1(x)|dx\right)^2}{|\psi(x)|^2}dx = \frac{c}{N}\Vert g_1\Vert_1^2\int_{-\pi l}^{\pi l}\frac{1}{|\psi(x)|^2}dx.\nonumber
		\end{align}
	Part (III) can be estimated by Ljapunov inequality and Lemma \ref{Le1} as 
	\begin{align*}
	&\E|\hat{\theta}(x)-\theta(x)|^2 
	 \leq \left( \E|\hat{\theta}(x)-\theta(x)|^4\right)^{1/2}   \leq \frac{\sqrt{C}}{N} \sqrt{\E|Y_0|^4}.
	\end{align*}
	So putting all these results together, it follows for some constant $K>0$ that
	\begin{align*}
	&  \lVert g_1-\hat{g}_{1,l}\rVert_\cdot^2\leq \lVert g_1-g_{1,l} \rVert_2^2 + \frac{1}{\pi}\left[\frac{2\sqrt{c_1\cdot c_2}}{N}\sqrt{\E|Y_0|^4}\int_{-\pi l}^{\pi l}\frac{1}{|\psi(x)|^2}dx \right. \\
	&\qquad \left. + \frac{2c}{N} \Vert g_1\Vert_1^2\int_{-\pi l}^{\pi l}\frac{1}{|\psi(x)|^2}dx + \frac{\sqrt{C}}{N}\sqrt{ \E|Y_0|^4} \int_{-\pi l}^{\pi l}\frac{1}{|\psi(x)|^2}dx\right]\\
	&\leq \lVert g_1-g_{1,l} \rVert_2^2 + \frac{K}{N}\left(\sqrt{\E|Y_0|^4}+ \lVert g_1 \rVert_1^2 \right) \int_{-\pi l}^{\pi l}\frac{1}{|\psi(x)|^2}dx,
	\end{align*}
	that completes the proof.
\end{proof}

\begin{proof}[Proof of Corollary \ref{cor:ErrEst1}]

 Consider expression \eqref{eq:F_psi} in the proof of Theorem \ref{TheoGenCat}. Using assumptions  {\bf (H3)}-- {\bf (H4)}   there, 
	 it holds		
	$$
	\int_{-\pi l}^{\pi l}\frac{|\F[g_1](x)|^2}{|\psi(x)|^2}dx \leq  \frac{1}{c_\psi^2}\int_{\R}|\F[g_1](x)|^2(1+x^2)^{\beta}dx \leq \frac{L}{c_{\psi}^2}.
	$$
\end{proof}

\begin{proof}[Proof of Corollary \ref{cor:ErrEst2}]
Since $\cal F$ is an isometry of $L^2(\R)$ one has 
	\begin{align*}
	&  \|  g-g_{1,l} \|_2^2=    \|  \F[g_1] \id\{ |\cdot|>\pi l  \} \|_2^2  = \int_{|x|>\pi l }|\F[g_1](x)|^2(1+x^2)^{\beta}(1+x^2)^{-\beta} dx \\
	& \leq \max_{|x|>\pi l }{(1+x^2)^{-\beta}}\int_{\R}|\F[g_1](x)|^2(1+x^2)^{\beta}dx \leq \frac{L}{(1+(\pi l )^2)^{\beta}}
	\end{align*}
	by assumption  {\bf (H4)}. Using  {\bf (H3)} one gets 	
	$$
	\int_{-\pi l}^{\pi l}\frac{dx}{|\psi(x)|^2} \leq  c_\psi \int_{-\pi l}^{\pi l}(1+x^2)^{\beta}dx \leq 2c_{\psi} \pi l \big(1+(\pi l)^2\big)^{\beta}.
	$$
	Plugging this into \eqref{EqErrorGenCat} yields the result.
\end{proof}


\section*{Acknowledgement}
W. Karcher and E. Spodarev are grateful to E. V. Jensen for her hospitality during their stay at Aarhus University in February 2011 where this research was initiated. The authors thank M. Rei\ss \ for the fruitful discussions on the subject of the paper. They also acknowledge the valuable help of O. Moreva
in implementing the algorithms of Section \ref{sect:Num}.


\bibliographystyle{unsrt} 
\bibliography{bibliography}

\end{document}